\documentclass[12pt,A4paper]{article}

\usepackage[left=32mm,right=32mm,top=30mm,bottom=32mm]{geometry}
\usepackage{labelfig}
\usepackage{epsfig}
\usepackage{epstopdf}

\usepackage{xfrac}

\usepackage{color}
\usepackage{amsthm,amsmath,amssymb}
\usepackage{booktabs}
\usepackage{txfonts}
\usepackage{mathpazo}
\usepackage{microtype}
\usepackage{overpic}
\usepackage{bm}
\usepackage{sectsty}
\usepackage[
	pdftitle={PDFTitle},
	pdfauthor={Hugo Parlier},
	ocgcolorlinks,
	linkcolor=linkred,
	citecolor=linkred,
	urlcolor=linkblue]
{hyperref}

\definecolor{linkred}{rgb}{0.48,0.1,0.05}
\definecolor{linkblue}{RGB}{16, 78, 139}

\usepackage[hang,flushmargin]{footmisc}
\usepackage{enumitem}
\usepackage{titlesec}
	\titlespacing{\section}{0pt}{12pt}{0pt}
	\titlespacing{\subsection}{0pt}{6pt}{0pt}
	
\titlelabel{\thetitle.\quad}

\makeatletter 

\long\def\@footnotetext#1{%
\H@@footnotetext{%
\ifHy@nesting 
\hyper@@anchor{\@currentHref}{#1}%
\else 
\Hy@raisedlink{\hyper@@anchor{\@currentHref}{\relax}}#1%
\fi 
}}

\def\@footnotemark{%
\leavevmode 
\ifhmode\edef\@x@sf{\the\spacefactor}\nobreak\fi 
\H@refstepcounter{Hfootnote}%
\hyper@makecurrent{Hfootnote}%
\hyper@linkstart{link}{\@currentHref}%
\@makefnmark 
\hyper@linkend 
\ifhmode\spacefactor\@x@sf\fi 
\relax 
}%

\ifFN@multiplefootnote%
\renewcommand*\@footnotemark{%
\leavevmode 
\ifhmode 
\edef\@x@sf{\the\spacefactor}%
\FN@mf@check 
\nobreak 
\fi 
\H@refstepcounter{Hfootnote}%
\hyper@makecurrent{Hfootnote}%
\hyper@linkstart{link}{\@currentHref}%
\@makefnmark 
\hyper@linkend 
\ifFN@pp@towrite 
\FN@pp@writetemp 
\FN@pp@towritefalse 
\fi 
\FN@mf@prepare 
\ifhmode\spacefactor\@x@sf\fi 
\relax%
}%
\fi 

\makeatother 

\theoremstyle{plain}
\newtheorem{theorem}{Theorem}[section]
\newtheorem{proposition}[theorem]{Proposition}
\newtheorem{lemma}[theorem]{Lemma}
\newtheorem{corollary}[theorem]{Corollary}
\newtheorem{conjecture}[theorem]{Conjecture}

\theoremstyle{definition}

\newcommand{\Z}{{\mathbb Z}}

\newcommand{\diam}{{\rm diam}}

\newcommand{\contract}{\mathord{\varparallelinv}}

\newcommand{\FS}{\mathcal F {(\Sigma_n)}}

\newcommand{\FCO}{\mathcal F{ (\Gamma_1)}}
\newcommand{\MF}{\mathcal M \mathcal F}
\newcommand{\MFS}{{\mathcal M \mathcal F} (\Sigma_n)}

\newcommand{\MFC}{{\mathcal M \mathcal F} (\Gamma_n)}

\newcommand{\mcg}{{\rm Mod}}

\sectionfont{\large \sc}
\subsectionfont{\normalsize}

\long\def\symbolfootnote[#1]#2{\begingroup%
\def\thefootnote{\fnsymbol{footnote}}\footnote[#1]{#2}\endgroup}

\def\blfootnote{\xdef\@thefnmark{}\@footnotetext}

\begin{document}

{\Large \bfseries \sc Flip-graph moduli spaces of filling surfaces}

{\bfseries Hugo Parlier\symbolfootnote[1]{\normalsize Research supported by Swiss National Science Foundation grant PP00P2\textunderscore 128557
}, Lionel Pournin\symbolfootnote[2]{\normalsize Research funded by the ANR project IComb (grant ANR-08-JCJC-0011)
}}

\vspace{0.2cm}
{\em Abstract.}

This paper is about the geometry of flip-graphs associated to triangulations of surfaces. More precisely, we consider a topological surface with a privileged boundary curve and study the spaces of its triangulations with $n$ vertices on the boundary curve. The surfaces we consider topologically fill this boundary curve so we call them {\it filling surfaces}. The associated flip-graphs are infinite whenever the mapping class group of the surface (the group of self-homeomorphisms up to isotopy) is infinite, and we can obtain moduli spaces of flip-graphs by considering the flip-graphs up to the action of the mapping class group. This always results in finite graphs and we are interested in their geometry. 

Our main focus is on the diameter growth of these graphs as $n$ increases. We obtain general estimates that hold for all topological types of filling surface. We find more precise estimates for certain families of filling surfaces and obtain asymptotic growth results for several of them. In particular, we find the exact diameter of modular flip-graphs when the filling surface is a cylinder with a single vertex on the non-privileged boundary curve. 

\section{Introduction}\label{sec:intro}

Triangulations of surfaces are very natural objects that appear in the study of topological, geometrical, algebraic, probabilistic and combinatorial aspects of surfaces and related topics. We are interested in a natural structure on spaces of triangulations -- flip-graphs. Vertices of flip-graphs are triangulations and two triangulations span an edge if they differ only by a single arc (our base surface is a topological object and we consider triangulations up to vertex preserving isotopy). When edge lengths are all set to one, flip-graphs are geometric objects which provide a measure for how different triangulations can be.

Flip-graphs appear in different contexts and take different forms. As flipping an arc (replacing an arc by another one) does not change either the vertices or the topology of the surface, flip-graphs correspond to triangulations of homeomorphic surfaces with a prescribed set of vertices. Provided the surface has enough topology, flip-graphs are infinite, and self homeomorphisms of the surface act on this graph as isomorphisms. In fact, modulo some exceptional cases, the {\it mapping class group} of the surface or the group of self-homeomorphisms of the surface up to isotopy is exactly the automorphism group of the graph \cite{KorkmazPapadopoulos2012}. The quotient of a flip-graph via its automorphism group is finite, and thus via the Svarc-Milnor Lemma (see for example \cite{BridsonHaefliger1999}), a flip-graph and the associated mapping class group are quasi-isometric. 

Furthermore, if one gives the triangles in a triangulation a given geometry, each triangulation corresponds to a geometric structure on a surface. In this direction, Brooks and Makover \cite{BrooksMakover2004} defined {\it random surfaces} to be geometric surfaces coming from a random triangulation where each triangle is an ideal hyperbolic triangle. This notion of a random surface is a way of sampling points in {\it Teichm\"uller and moduli spaces} - roughly speaking the space of hyperbolic metrics on a given topological structure. Although in the above it is only the vertex set of flip-graphs that appear, in the theory of {\it decorated} Teichm\"uller spaces, flip-graphs play an integral role \cite{Penner1987}. In a similar direction, Fomin, Shapiro, and Thurston \cite{FominShapiroThurston2008}, and more recently Fomin and Thurston \cite{FominThurston2012}, have used flip-graphs and variants to study cluster algebras that come from the Teichm\"uller theory of bordered surfaces. 

For all of these reasons, flip-graphs and their relatives appear frequently and importantly in the study of moduli spaces, surface topology and the study of mapping class groups. 

In a different context, flip-graphs are important objects for the study of triangulations of arbitrary dimension, whose vertices are placed in a Euclidean space and whose simplices are embedded linearly (see \cite{DeLoeraRambauSantos2010} and references therein). In this case, flip-graphs are always finite, and they are sometimes isomorphic to the graph of a polytope, or admit subgraphs that have this property. Such flip-graphs emerge for instance from the study of generalized hypergeometric functions and discriminants \cite{GelfandKapranovZelevinsky1990} and from the theory of cluster algebras \cite{FominZelevinsky2003}. The simplest non-trivial case is that of the flip-graph a polygon with $n$ vertices, which turns out to be the graph of a celebrated polytope - the associahedron \cite{Lee1989}. The study of this graph has an interesting history of its own \cite{Stasheff2012}, and one of the reasons it has attracted so much interest is that it pops up in surprisingly different contexts (see for instance \cite{Lee1989,SleatorTarjanThurston1988,Stasheff1963,Tamari1954}).

Associahedra appear, in particular, in the work of Sleator, Tarjan and Thurston on the dynamic optimality conjecture \cite{SleatorTarjanThurston1988}. They proved the theorem below about the diameter of these polytopes for sufficiently large $n$, using constructions of polyhedra in hyperbolic $3$-space \cite{SleatorTarjanThurston1988}. Their proof, however, does not tell how large $n$ should be for the theorem to hold. The second author proved this theorem whenever $n$ is greater than $12$ using combinatorial arguments \cite{Pournin2014}. Note that for smaller $n$ the diameter behaves differently. 

{\bf Theorem \cite{Pournin2014}.} {\it The flip-graph of a convex polygon with $n$ vertices has diameter $2n-10$ whenever $n$ is greater than $12$.}

This theorem is in some sense our starting point. The topology of a polygon is the simplest that one can imagine - it is simply the boundary circle filled by a disk. Our basic question is as follows - what happens when one replaces the disk by a surface with more topology? These {\it filling} surfaces (as they fill the boundary circle) generally give rise to infinite flips graphs. More precisely, unless the filling surface is a disk, a disk with a single marked point or a M\"obius band, the associated flip-graph is infinite. Up to homeomorphism which preserves the circle boundary pointwise however, we get a nice finite combinatorial moduli space of triangulations - so the question of bounding its diameter makes sense. 

Precise definitions and notations can be found in the next section - but in order to state our results we briefly describe our notation here. $\Sigma$ is the filling surface (so a topological surface with a privileged boundary curve) and $\Sigma_n$ is the same surface with $n$ marked points on the privileged boundary. The modular flip-graph $ \MFS$ is the flip-graph up to homeomorphism of $\Sigma_n$. For example $\MFS$ is the graph of the associahedron when $\Sigma$ is a disk.

Our first result is the following upper bound for the diameter which does not asymptotically depend on the topology of the filling surface. 

\begin{theorem}\label{thm:basicupperboundintro}
For any $\Sigma$ there exists a constant $K_{\Sigma}$ such that 
$$
\diam(\MFS)\leq 4 n + K_{\Sigma}.
$$
\end{theorem}

A simple consequence of this result and of the monotonicity of $\diam(\MFS)$, proven in Section \ref{Sec.projlemma}, is that the limit
$$
\lim_{n\to \infty} \frac{\diam(\MFS)}{n} 
$$
exists (and is less than or equal to $4$). Again, in the case of the associahedron, this limit is $2$. It is perhaps not a priori obvious why the limit should not {\it always} be $2$, independently of the topology of $\Sigma$, but this turns out not to be the case. 

In order to exhibit different behaviors, we study particular examples of $\Sigma$. Our examples are surfaces $\Sigma$ with genus $0$ and $k+1$ boundaries, including the privileged one, and each of the non-privileged boundaries has a single vertex. In the sequel we will refer to these non-privileged boundary curves with a single vertex as {\it boundary loops}. They can be marked or unmarked - this corresponds to disallowing or allowing the mapping class group acting on the flip-graph to exchange the loops. We provide the following upper bounds for the diameters.
\begin{theorem}\label{thm:uppermarkedintro}
Let $\Sigma$ be a filling surface with genus $0$ and $k$ marked boundary loops. Then there exists a constant $K_k$ which only depends on $k$ such that
$$
\diam(\MFS) \leq \left(4-\frac{2}{k}\right) \, n + K_k.
$$
\end{theorem}
Similarly:
\begin{theorem}\label{thm:upperunmarkedintro}
Let $\Sigma$ be a filling surface with genus $0$ and $k$ unmarked boundary loops. Then there exists a constant $K_k$ which only depends on $k$ such that
$$
\diam(\MFS) \leq \left(3-\frac{1}{2k}\right) \, n + K_k.
$$
\end{theorem}
In the case of the associahedron, upper bounds of the correct order ($2n$) are somewhat immediate, but here the upper bounds, although not particularly mysterious, are somewhat more involved.

Of course if there is only one loop, being marked or unmarked doesn't matter and we denote this surface $\Gamma$. In the case of $\Gamma$, we are able to prove a precise result about the diameter of the flip-graph, which in particular shows that the upper bound in the previous theorem is asymptotically sharp for $k=1$ at least.
\begin{theorem}\label{thm:gammaintro}
The diameters of the modular flip-graphs of $\Gamma$ satisfy
$$
\diam \left( \MFC \right) =\left\lfloor{ \frac{5}{2} n }\right\rfloor-2.
$$
\end{theorem}
As for the associahedron, the hard part is the lower bound. We note in the final section that the lower bound from this theorem proves general lower bound provided $\Sigma$ has at least one interior marked point and enough topology. Our final main result is when $\Sigma$ has exactly two marked boundary loops - we call this particular surface $\Pi$. We prove the following.
\begin{theorem}\label{thm:piintro}
The diameter of $\mathcal{MF}(\Pi_n)$ is not less than $3n$.
\end{theorem}
This result and the upper bound from Theorem \ref{thm:uppermarkedintro} when $k=2$ show that the diameters grow like $3n$ (with constant error term).

Our lower bounds always come from somewhat involved combinatorial arguments, using the methods introduced in \cite{Pournin2014}. Boundary loops play an important part as to ensure that two triangulations are far apart, we show that moving these loops necessarily entails a certain number flips. 

This article is organized as follows. We begin with a section devoted to preliminaries which include notation and basic or previous results we need in the sequel. As the results may be of interest to people with different mathematical backgrounds, we spend some time talking about the setup in order to keep it as self contained as possible. The third section is about upper bounds and the fourth and fifth about lower bounds. In the final section, we discuss some consequences of our results and we conclude with several questions and conjectures about what the more general picture might look like.

\section{Preliminaries}

In this section we describe in some detail the objects we are interested in, introduce notations and some of the tools we use in the sequel. In particular, the methods used in \cite{Pournin2014} to obtain lower bounds on the diameter of flip-graphs are generalized in Subsection \ref{Asubsection.2.2}.

\subsection{Basic setup}
Our basic setup is as follows. We begin with a topological orientable surface $\Sigma$ with the following properties. 

{\underline{Property 1}}: $\Sigma$ has at least one boundary curve, and we think of this boundary curve as being special. We will refer to this special boundary curve as being the {\it privileged boundary curve}. (It has no marked or unmarked points on it, but will be endowed with them in what follows.)

{\underline{Property 2}}: All non privileged boundary curves of $\Sigma$ have at least one marked or unmarked point on it. This is because we will want to triangulate $\Sigma$ and these points will be an integral part of the vertices of the triangulation. The distinction between marked and unmarked points will become clearer in the following, but note that if a boundary curve contains one marked point, all points on the boundary are naturally marked as they are determined by their relative position to the marked point on the boundary. Also note that most of the specific examples we study in more detail have only marked points. 

{\underline{Property 3}}: $\Sigma$ is of finite type. It can have genus, marked or unmarked points in its interior or on its non-privileged boundary curves, but only a finite number of each. Another way of saying this is by asking that its group of self homeomorphisms be finitely generated (but not necessarily finite). 

We illustrate $\Sigma$ in Fig. \ref{Hfigure.surface} with its different possible features. 
\begin{figure}
\begin{centering}
\includegraphics{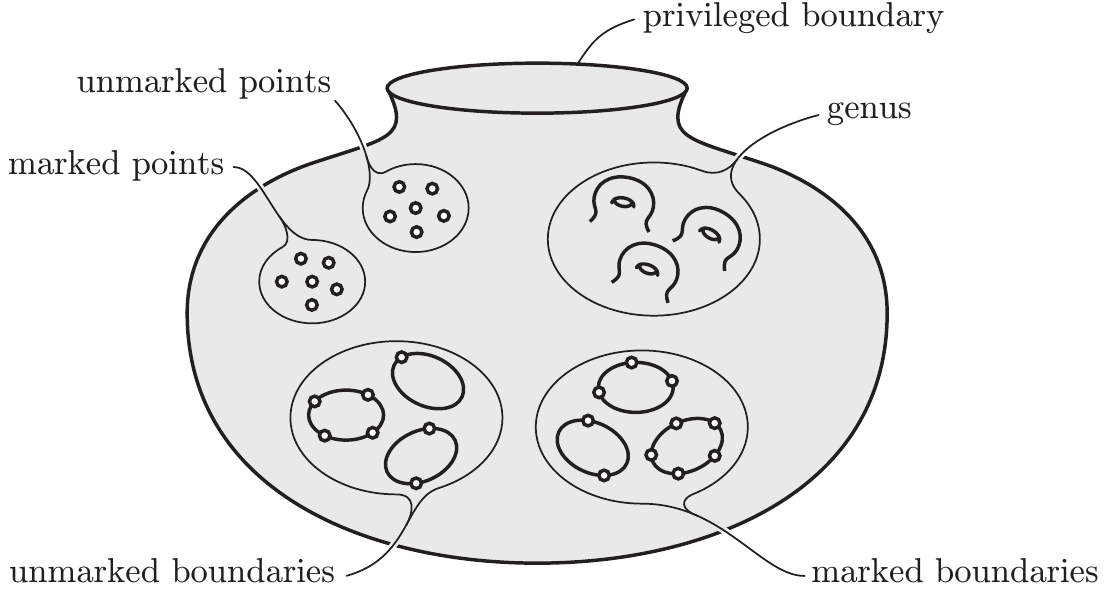}
\caption{$\Sigma$ and its possible features}
\label{Hfigure.surface}
\end{centering}
\end{figure}
Note that, if it has no topology, then $\Sigma$ is simply a disk. 

For any positive integer $n$, from $\Sigma$ we obtain a surface $\Sigma_n$ by placing $n$ marked points on the privileged boundary of $\Sigma$. We are interested in triangulating $\Sigma_n$ and studying the geometry of the resulting flip-graphs.
We fix $\Sigma_n$, and we refer to its set of marked and unmarked points as its {\it vertices}. An {\it arc} of $\Sigma_n$ is an isotopy class of non-oriented simple paths between two vertices (non-necessarily distinct). 

From arcs one can construct a simplicial complex called the arc complex. This complex is well studied in geometric topology; it is built by associating simplices to sets of arcs that can be realized disjointly. A {\it triangulation} of $\Sigma$ is a maximal collection of arcs that can be realized disjointly. Although they are not necessarily ``proper" triangulations in the usual sense, they do cut the surface into a collection of triangles. 

For fixed $\Sigma_n$, the number of interior arcs of a triangulation is a fixed number and we call this number the {\it arc complexity} of $\Sigma_n$. Note that by an Euler characteristic argument it increases linearly in $n$. 

We now construct the {\it flip-graph} $\FS$ as follows. Vertices of $\FS$ are the triangulations of $\Sigma_n$, and two vertices share an edge if they coincide in all but one arc. Another way of seeing this is that they share an edge if they are related by a {\it flip}, as shown on Fig. \ref{fig:flip}.
\begin{figure}
\begin{centering}
\includegraphics{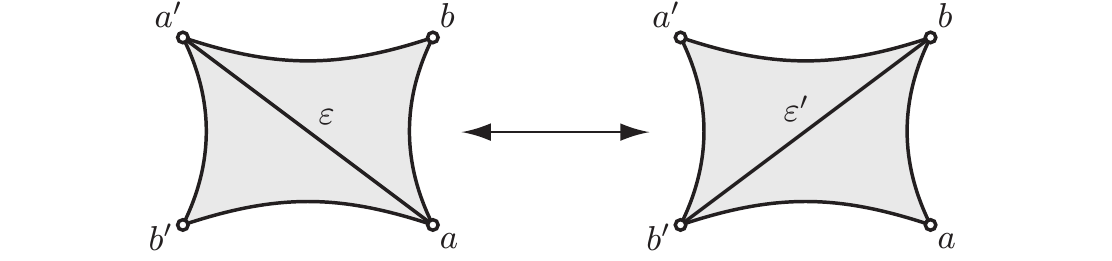}
\caption{The flip that exchanges arcs $\varepsilon$ and $\varepsilon'$}
\label{fig:flip}
\end{centering}
\end{figure}
This graph is sometimes finite, sometimes infinite, but it is always connected and any isotopy class of arc can be introduced into a triangulation by a finite number of flips (see for instance \cite{Mosher1988}).

When $\Sigma$ is a disk, $\FS$ is finite and is the graph of the associahedron. A simple example of an infinite flip-graph is when $\Sigma$ is a surface with one boundary loop (in addition to the privileged boundary) and no other topology. We denote this surface $\Gamma$ for future reference. The simplest case is $\Gamma_1$ and a triangulation of $\Gamma_1$ always contains two interior arcs, both between distinct marked points. Each triangulation can be flipped in two ways, so $\FCO$ is everywhere of degree 2. Furthermore it is straightforward to see that it is infinite and connected so in fact is isomorphic to the infinite line graph ($\Z$ with its obvious graph structure). 

In the event that $\FS$ is infinite, there is a non-trivial natural action of the group of self-homeomorphisms of $\Sigma_n$ on $\FS$. This is because homeomorphisms will preserve the property of two triangulations being related by a flip so they induce a simplicial action on $\FS$. This is where the importance of being a marked or an unmarked point plays a part. We allow homeomorphisms to exchange unmarked points (but fix them globally as a set). In contrast they must fix all marked points individually. We denote $\mcg(\Sigma_n)$ the group of such homeomorphisms up to isotopy. Note that once $n\geq 3$, by the action on the privileged boundary of $\Sigma$, all such homeomorphisms are orientation preserving. As we are interested primarily in large $n$, we don't need to worry about orientation reversing homeomorphisms. 

The combinatorial moduli spaces we are interested in are thus 
$$
\MFS := \FS / \mcg (\Sigma_n).
$$

Observe that this always gives rise to connected finite graphs. To unify notation, we also denote the corresponding flip-graph by $\MFS$ even if homeomorphism group action is trivial.

We think of these graphs as discrete metric spaces where points are vertices of the graphs and the distance is the usual graph distance with edge length $1$. In particular, some of these graphs have loops (a single edge from a vertex to itself) but adding or removing a loop gives rise to an identical metric space. We think of these graphs as not having any loops.

Our main focus is on the (vertex) diameter of these graphs, which we denote $\diam(\MFS)$ and how these grow in function of $n$ for fixed $\Sigma$. In order to exhibit maximally distant triangulations, we will spend some time studying particular topological types of $\Sigma$. One of them is $\Gamma$, already described above. It has one boundary loop (recall that a boundary loop refers to a non-privileged boundary with a single vertex). Similarly we shall consider $\Pi$ which has exactly two boundary loops and no other topology.

\subsection{Deleting a vertex of the privileged boundary}
\label{Asubsection.2.2}

One of the main ingredients used in \cite{Pournin2014} to obtain lower bounds on flip distances is the operation of deleting a vertex from a triangulation. Here, we will use this operation to the same end. Vertices of the privileged boundary will be deleted from triangulations of a given surface $\Sigma_n$, resulting in triangulations of $\Sigma_{n-1}$ when $n$ is greater than $1$.

For a surface $\Sigma_n$ we label the vertices on the privileged boundary $a_1$ to $a_n$ in such a way that two vertices with consecutive indices are also consecutive on the boundary. Furthermore, the boundary arc with vertices $a_p$ and $a_{p+1}$ will be denoted $\alpha_p$, and the boundary arc with vertices $a_n$ and $a_1$ by $\alpha_n$.

Now consider a triangulation $T$ of $\Sigma_n$. Some triangle $t$ of $T$, depicted on the left of Fig. \ref{Afigure.2.0}, is incident to arc $\alpha_p$.
\begin{figure}
\begin{centering}
\includegraphics{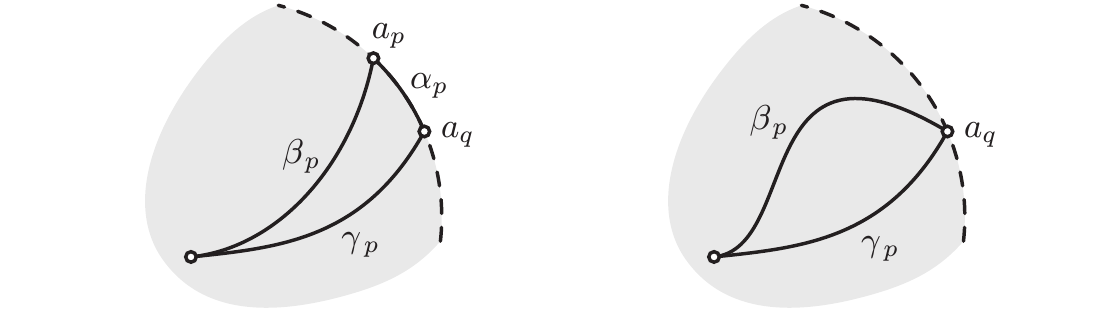}
\caption{The triangle incident to arc $\alpha_p$ in some triangulation of $\Sigma_n$ (left), and what happens to it when vertex $a_p$ is displaced to the other vertex of $\alpha_p$ along the boundary (right).}\label{Afigure.2.0}
\end{centering}
\end{figure}
Assuming that $n$ is greater than $1$, this triangle necessarily has two other distinct edges. Denote these edges by $\beta_p$ and $\gamma_p$ as shown on the figure. Deleting vertex $a_p$ consists in displacing this vertex along the boundary to the other vertex of $\alpha_p$, and by removing arc $\beta_p$ from the resulting set of arcs. Observe in particular that the displacement of vertex $a_p$ removes $a_p$ from the privileged boundary and arc $\alpha_p$ from the triangulation as shown on the right of Fig. \ref{Afigure.2.0}. Moreover, arcs $\beta_p$ and $\gamma_p$ have then become isotopic, and the removal of arc $\beta_p$ results in a triangulation of $\Sigma_{n-1}$.

Note that the deletion operation preserves triangulation homeomorphy. Therefore, this operation carries over to the moduli of flip-graphs and transforms any triangulation in $\MF(\Sigma_n)$ into a triangulation that belongs to $\MF(\Sigma_{n-1})$. The triangulation obtained by deleting vertex $a_p$ from $T$ is called $T\contract{p}$ in the remainder of the paper, following the notation introduced in \cite{Pournin2014}. This notation will be used indifferently whether $T$ is a triangulation of $\Sigma_n$ or belongs to $\MF(\Sigma_n)$.

Consider two triangulations $U$ and $V$ in $\MF(\Sigma_n)$ and assume that they can be obtained from one another by a flip. The following proposition shows that the relation between $U\contract{p}$ and $V\contract{p}$ can be of two kinds.
\begin{proposition}\label{Aproposition.2.0}
Suppose $n\geq 2$. If $U$ and $V$ are triangulations in $\MF(\Sigma_n)$ related by a flip, then $U\contract{p}$ and $V\contract{p}$ are either identical or they are related by a flip.
\end{proposition}
\begin{proof}
Consider the quadrilateral whose diagonals are exchanged by the flip relating $U$ and $V$. The deletion of vertex $a_p$ either shrinks this quadrilateral to a triangle, deforms it to another quadrilateral, or leaves this quadrilateral unaffected. In the former case, $U\contract{p}$ and $V\contract{p}$ are identical because the deletion then removes the two arcs exchanged by the flip. In the other two cases, $U\contract{p}$ and $V\contract{p}$ can also be identical (while vertex deletion preserves homeomorphy, it does not always preserve non-homeomorphy), but if they are not, they differ exactly on the (possibly deformed) quadrilateral. More precisely, they can be obtained from one another by the flip that exchanges the diagonals of this quadrilateral.
\end{proof}

In the sequel, a flip between two triangulations $U$ and $V$ in $\MF(\Sigma_n)$ is called \emph{incident to arc $\alpha_p$} when $U\contract{p}$ is identical to $V\contract{p}$.

When $\Sigma$ is a disc, the flips incident to arc $\alpha_p$ are exactly the ones that affect the triangle incident to this arc within a triangulation \cite{Pournin2014}. When $\Sigma$ is not a disc, these flips are still incident to $\alpha_p$, but they are not necessarily the only ones. For instance, the unique triangulation in $\MF(\Gamma_1)$ and the four triangulations in $\MF(\Gamma_2)$ are depicted in Fig. \ref{Afigure.2.1}. Since $\MF(\Gamma_1)$ has a single element, we have:
\begin{figure}
\begin{centering}
\includegraphics{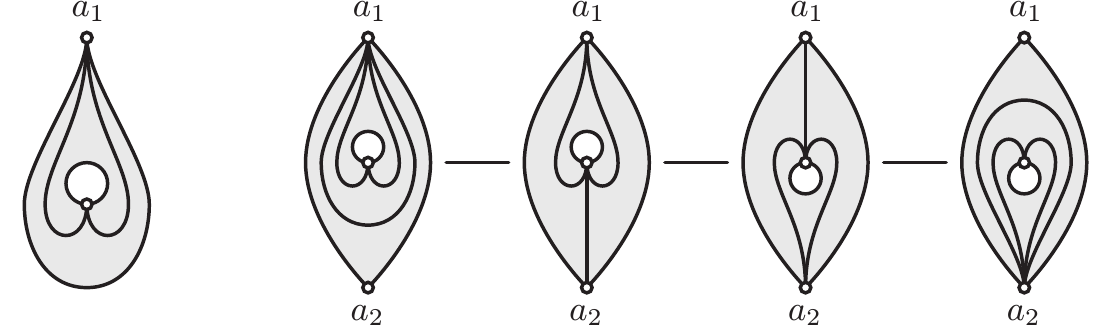}
\caption{The unique triangulation in $\MF(\Gamma_1)$ (left) and the four triangulations in $\MF(\Gamma_2)$. The lines between the latter four triangulations depict $\mathcal{MF}(\Gamma_2)$.}\label{Afigure.2.1}
\end{centering}
\end{figure}

\begin{proposition}\label{Aproposition.2.1}
If $T$ is one of the four triangulations in $\MF(\Gamma_2)$, then any flip carried out in $T$ is incident to both $\alpha_1$ and $\alpha_2$.
\end{proposition}

The modular flip-graph of $\Gamma_2$ is shown in Fig. \ref{Afigure.2.1}. In this flip-graph, the third triangulation from the left is obtained from the second one by replacing any of the two interior arcs incident to $a_1$ by an interior arc incident to $a_2$. Assume that the removed arc is the one on the left. In this case, the triangle incident to $\alpha_1$ is not affected by the flip. Yet, this flip is incident to $\alpha_1$ because of Proposition \ref{Aproposition.2.1}.

Now assume that $U$ and $V$ are two arbitrary triangulations that belong to $\MF(\Sigma_n)$. Consider a sequence $(T_i)_{0\leq{i}\leq{k}}$ of triangulations in $\MF(\Sigma_n)$ so that $T_0=U$, $T_k=V$, and $T_{i-1}$ can be transformed into $T_i$ by a flip whenever $0<i\leq{k}$. Such a sequence will be called a path of length $k$ from $U$ to $V$, and can be alternatively thought of as a sequence of flips that transform $U$ into $V$. According to Proposition \ref{Aproposition.2.0}, removing unnecessary triangulations from the sequence $(T_i\contract{p})_{0\leq{i}\leq{k}}$ results in a path from $U\contract{p}$ to $V\contract{p}$ and the number of triangulations that need be removed from the sequence is equal to the number of flips incident to $\alpha_p$ along $(T_i)_{0\leq{i}\leq{k}}$. In other words:

\begin{lemma}\label{Alemma.2.1}
Let $U$ and $V$ be two triangulations in $\MF(\Sigma_n)$. If $f$ flips are incident to arc $\alpha_p$ along a path of length $k$ between $U$ and $V$, then there exists a path of length $k-f$ between $U{\contract}p$ and $V\mathord{\contract}p$.
\end{lemma}

Note that when $\Sigma$ is a disc, this lemma is exactly Theorem 3 from \cite{Pournin2014}. A path between two triangulations $U$ and $V$ in $\MF(\Sigma_n)$ is called geodesic if its length is minimal among all the paths between $U$ and $V$. The length of any such geodesic is equal to the distance of $U$ and $V$ in $\mathcal{MF}(\Sigma_n)$, denoted by $d(U,V)$. Invoking Lemma \ref{Alemma.2.1} with a geodesic between $U$ and $V$ immediately yields:

\begin{theorem}\label{Atheorem.2.3}
Let $n$ be an integer greater than $1$. Further consider two triangulations $U$ and $V$ in $\MF(\Sigma_n)$. If there exists a geodesic between $U$ and $V$ along which at least $f$ flips are incident to arc $\alpha_p$, then the following inequality holds:
$$
d(U,V)\geq{d(U{\contract}p,V{\contract}p)+f}\mbox{.}
$$
\end{theorem}

In well defined situations, at least two flips along any geodesic are incident to a given boundary arc. This may be the case when one of the triangulations at the ends of the geodesic has a well placed \emph{ear}, i.e. a triangle with two edges in the privileged boundary as shown on the left of Fig. \ref{Afigure.2.05}. In the figure, these two edges are $\alpha_p$ and $\alpha_q$, and the vertex they share is $a_q$. In this case, we will say that the triangulation has an ear \emph{in $a_q$}. The following result, proven in \cite{Pournin2014} when $\Sigma$ is a disc, still works in the more general case at hand:

\begin{lemma}\label{Alemma.2.75}
Consider two triangulations $U$ and $V$ in $\MF(\Sigma_n)$. Further consider two distinct arcs $\alpha_p$ and $\alpha_q$ on the privileged boundary of $\Sigma_n$ so that $a_q$ is a vertex of $\alpha_p$. If $U$ has an ear in $a_q$ and if the triangles of $V$ incident to $\alpha_p$ and to $\alpha_q$ do not have a common edge, then for any geodesic between $U$ and $V$, there exists $r\in\{p,q\}$ so that at least two flips along this geodesic are incident to $\alpha_r$.
\end{lemma}
\begin{proof}
Assume that $U$ has an ear in $a_q$ and that the triangles of $V$ incident to $\alpha_p$ and to $\alpha_q$ do not have a common edge. In this case, $U$ and $V$ are as shown respectively on the left and on the right of Fig. \ref{Afigure.2.05}. Note that vertices $b$ and $c$ represented in this figure can be identical. At least one flip along any path between $U$ and $V$ is incident to arc $\alpha_p$ because the triangles of $U$ and of $V$ incident to this arc are distinct.

Consider a geodesic $(T_i)_{0\leq{i}\leq{k}}$ from $U$ to $V$ and assume that only one of the flips along this geodesic is incident to $\alpha_p$, say the $j$-th one. This flip must then be as shown in the center of Fig. \ref{Afigure.2.05}.
\begin{figure}
\begin{centering}
\includegraphics{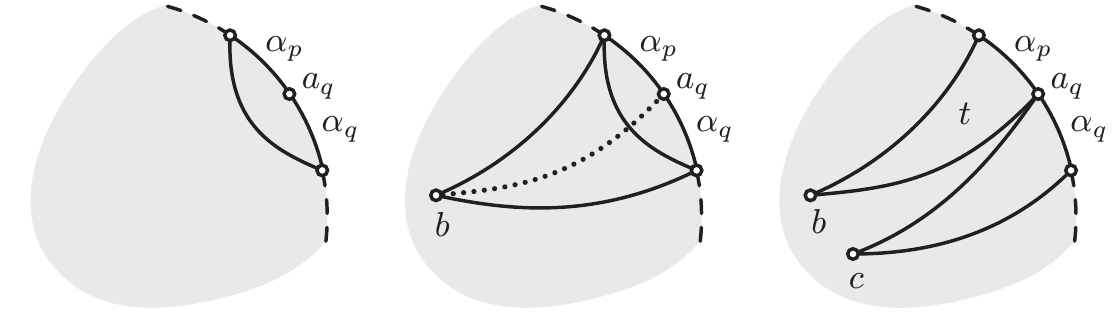}
\caption{Triangulations $U$ (left) and $V$ (right) from the statement of Lemma \ref{Alemma.2.75}. The $j$-th flip along the geodesic used in the proof of this lemma is shown in the center, where the solid edges belong to $T_{j-1}$, and the introduced edge is dotted.}\label{Afigure.2.05}
\end{centering}
\end{figure}
Not only is it incident to $\alpha_p$ but also to $\alpha_q$. Moreover, the triangle $t$ of $V$ incident to $\alpha_p$ already belongs to $T_j$. Now observe that the triangle of $T_j$ incident to $\alpha_q$ shares an edge with $t$. By assumption, the triangle of $V$ incident to $\alpha_q$ does not have this property. Therefore, at least one of the last $k-j$ flips along $(T_i)_{0\leq{i}\leq{k}}$ must affect the triangle incident to $\alpha_q$. This flip is then the second flip incident to $\alpha_q$ along the geodesic.
\end{proof}

\subsection{A projection lemma}\label{Sec.projlemma}

Here we briefly describe a result from \cite{DisarloParlier2014} in our setting and its implications on our diameter estimates. This lemma is about two triangulations $U$ and $V$ of $\Sigma_n$ with arcs in common. It tells that these arcs must also be arcs of all the triangulations on any geodesic between $U$ and $V$ in flip-graph $\mathcal{F}(\Sigma_n)$. This generalizes Lemma 3 from \cite{SleatorTarjanThurston1988}, originally proven in the case of a disc with marked boundary points. Formally:

\begin{lemma}[Projection Lemma]\label{lem:projlemma}
Let $U$ and $V$ be two triangulations of $\Sigma_n$. Further consider a geodesic $(T_i)_{0\leq{i}\leq{k}}$ from $U$ to $V$ in graph $\mathcal{F}(\Sigma_n)$. If $\mu$ is a multi arc common to $U$ and to $V$, then $\mu$ is also an arc of $T_i$ whenever $0<i<k$. 
\end{lemma}

It is absolutely essential to note that the above lemma does {\it not} necessarily hold in $\MFS$. However, it clearly does hold if an arc or a multi arc is invariant under all elements of $\mcg (\Sigma_n)$. Namely, consider an arc $\alpha$ {\it parallel} to the privileged boundary (by parallel we mean that the portion of $\Sigma_n$ bounded by this arc and by a part of the privileged boundary is a disc). Then, as any element of $\mcg (\Sigma_n)$ fixes the privileged boundary arcs individually, arc $\alpha$ is also invariant. In particular, assume that $\alpha$ has vertices $a_1$ and $a_3$. By the above, $\alpha$ is never removed along a geodesic between two triangulations containing this arc. So naturally we get a geodesically convex and isometric copy of $\MF (\Sigma_{n-1})$ inside $\MFS$. As such:

\begin{proposition}\label{prop:monotonicity}
$\diam(\MF (\Sigma_{n-1})) \leq \diam(\MFS).$
\end{proposition}

Note that, by observing that there are points outside the isometric copy of $\MF (\Sigma_{n-1})$, it is not too difficult to see that in fact the above inequality is strict but we make no particular use of that in the sequel.

\section{Upper bounds}
In this section we prove upper bounds on the diameter of modular flip-graphs vdepending on the topology of the underlying surface. 

\subsection{A general upper bound}
We begin with the following general upper bound. 
\begin{theorem}\label{thm:basicupperbound}
For any $\Sigma$ there exists a constant $K_{\Sigma}$ such that 
$$
\diam(\MFS)\leq 4 n + K_{\Sigma}.
$$
\end{theorem}

Before proving the theorem, let us give the basic idea of the proof. Consider a triangulation $T$ of $\Sigma_n$ and a vertex $a$ of this surface. Let us call the number of interior arcs of $T$ incident to $a$ the interior degree of $a$ in $T$. For large enough $n$ the average interior degree of the vertices of $T$ can be arbitrarily close to $2$ and thus given any two triangulations $U$ and $V$ the average sum of the interior degrees tends to $4$. We can then choose vertex $a$ (in the privileged boundary) in such a way that its interior degree is at most $4$. We perform flips within $U$ to obtain $\tilde{U}$ and flips within $V$ to obtain $\tilde{V}$ so that $\tilde{U}$ and $\tilde{V}$ both have an ear in $a$. In doing so we can now safely ignore a boundary vertex and repeat the process. 

To quantify how many flips each of the steps described above might cost, we prove the following lemma.

\begin{lemma}\label{Hlemma.3.1}
For $n\geq2$, consider a vertex $a$ in the privileged boundary of $\Sigma_n$ and two triangulations $U$ and $V$ of $\Sigma_n$. If the interior degrees of $a$ in $U$ and in $V$ sum to at most $4$, then there exist two triangulations $\tilde{U}$ and $\tilde{V}$ of $\Sigma_n$, each with an ear in $a$ so that
$$
d (U, \tilde{U}) + d (V, \tilde{V}) \leq 4.
$$
\end{lemma}
\begin{proof}
\begin{figure}[b]
\begin{centering}
\includegraphics{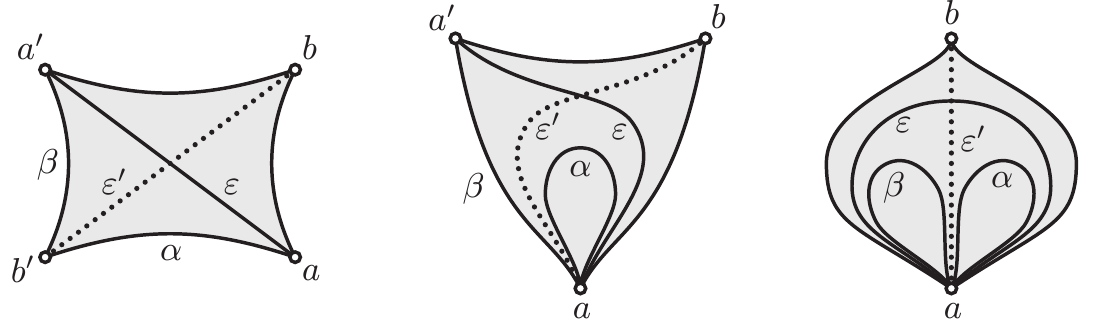}
\caption{The flip dealt with in the proof of Lemma \ref{Hlemma.3.1} (left), and a sketch of the surface when this flip does not reduce the degree of $a$ (center and right).}
\label{fig:LemmePrelim}
\end{centering}
\end{figure}

We shall prove the lemma by showing that there is always a flip in either $U$ or $V$ that reduces the degree of $a$, and thus by iteration, one must flip at most $4$ arcs to reach both $\tilde{U}$ and $\tilde{V}$.

Let $\varepsilon$ be any inner arc incident to $a$ in either $U$ or $V$.

First suppose that $\varepsilon$ is flippable. If flipping $\varepsilon$ reduces the degree of $a$, we flip it. If not, then necessarily the flip quadrilateral of $\varepsilon$ (shown on the left of Fig. \ref{fig:LemmePrelim}) must have a boundary arc, say $\alpha$, with vertex $a$ at its two ends. This situation, sketched in the center of Fig. \ref{fig:LemmePrelim} corresponds to when the vertex labeled $b'$ on the left of the figure is identical to $a$.

As $n$ is not less than $2$, $\alpha$ must be an interior arc. In addition, $\alpha$ is twice incident to $a$ and thus flippable. If flipping $\alpha$ reduces the degree of $a$, we flip $\alpha$ and we can proceed. So suppose flipping $\alpha$ does not decrease the degree of $a$. Then necessarily, the vertex $a'$ (as in Fig. \ref{fig:LemmePrelim}) is the same vertex as $a$. Arcs $\alpha,\beta$ and $\varepsilon$ (see the right side of Fig. \ref{fig:LemmePrelim}) are now three interior arcs twice incident to $a$. Thus the interior degree of $a$ is at least $6$ which is impossible.

Now consider the case where $\varepsilon$ is not flippable. Then it is surrounded by an arc $\varepsilon'$ twice incident to $a$ as in Fig. \ref{fig:LemmePrelim2}.
\begin{figure}
\begin{centering}
\includegraphics{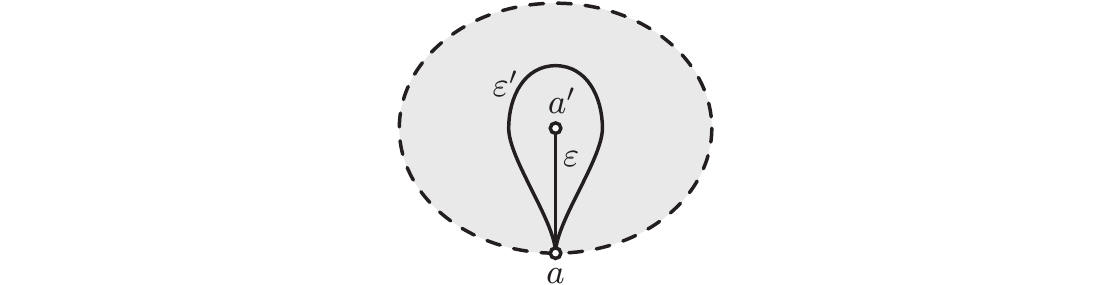}
\caption{When $\varepsilon$ is not flippable}
\label{fig:LemmePrelim2}
\end{centering}
\end{figure}
Flipping $\varepsilon'$ reduces the degree of $a$ because the flip introduces an arc incident to $a'$. 
\end{proof}

Note that Lemma \ref{Hlemma.3.1} holds a fortiori when $U$ and $V$ belong to $\MF(\Sigma_n)$. We can now prove the theorem. 

\begin{proof}[Proof of Theorem \ref{thm:basicupperbound}]
Consider surface $\Sigma_1$ and insert points in its privileged boundary to obtain $\Sigma_n$. The Euler characteristics satisfy
$$
\chi(\Sigma_n) = \chi(\Sigma_1).
$$
A triangulation $T$ of $\Sigma_n$ has $n-1$ more vertices and $n-1$ more triangles than a triangulation $T'$ of $\Sigma_1$. It also has $n-1$ more boundary arcs. By invariance of the Euler characteristic this means that $T$ has exactly $n-1$ more {\it interior} edges than $T'$. As such, the number of interior edges of $T$ is exactly 
$$
n + E_\Sigma\mbox{,}
$$
where $E_\Sigma$ is a precise constant which depends on $\Sigma$ but not on $n$. We now focus our attention on the interior degree of the privileged boundary vertices. The total interior degree of all vertices is $2 (n + E_\Sigma)$.

The sum of the interior degrees of all vertices in two triangulations $U$ and $V$ in $\MF(\Sigma_n)$ is $4(n + E_\Sigma)$. Thus the average sum of interior degrees among the privileged boundary vertices is at most
$$
4 + \frac{4}{n} E_\Sigma.
$$
As such, for $n > 4 E_\Sigma$, there exists a privileged boundary vertex $a$ whose interior degrees in $U$ and in $V$ sum to at most $4$.

We now apply the previous lemma to flip $U$ and $V$ a total of at most $4$ times into two new triangulations with ears in $a$. We treat the new triangulations as if they lay in $\MF(\Sigma_{n-1})$ and we repeat the process inductively until 
$n \leq 4 E_\Sigma$. We end up with two triangulations $\tilde{U}$ and $\tilde{V}$ that only differ on a subsurface homeomorphic to $\Sigma_{n_0}$, where
$$
n_0\leq 4 E_\Sigma.$$
Hence, there is a path of length at most $\diam(\MF(\Sigma_{n_0}))$ between $\tilde{U}$ and $\tilde{V}$. We can now conclude that
$$
d(U,V) \leq 4(n - 4 E_\Sigma) + \diam(\MF(\Sigma_{n_0})) = 4n + K_\Sigma,
$$
where $K_\Sigma$ does not depend on $n$.
\end{proof}

Before looking at more precise bounds for given surface topology, we note that together with the monotonicity from Proposition \ref{prop:monotonicity} we have the following:

\begin{corollary}\label{cor:limits}
For any $\Sigma$ the following limit exists and satisfies
$$
\lim_{n\to \infty}  \frac{ \diam(\MFS) }{n}\leq 4.
$$
\end{corollary}

\subsection{Upper bounds for $\Gamma$}

In this section we prove a much stronger and specific upper bound in the case where our surface is $\Gamma$, a cylinder with a single boundary loop.

\begin{theorem}\label{thm:gammaupper}
The diameter of the modular flip-graphs of $\Gamma$ satisfy
$$
\diam \left( \MFC \right) \leq \frac{5}{2} n -2.
$$
\end{theorem}

\begin{proof}
Let $U$ and $V$ be triangulations in $\MFC$. Denote by $a_0$ the unique vertex not on the privileged boundary and $\alpha_0$ the boundary loop it belongs to. The basic strategy is to perform flips within both triangulations until all interior arcs are incident to $a_0$ and then find a path between the resulting triangulations. 

We begin by observing that a triangulation in $\MFC$ has $n+1$ interior arcs. Furthermore, any triangulation $T$ of $\Gamma_n$ has at least $2$ distinct interior arcs incident to $a_0$. Indeed, $\alpha_0$ is incident to a triangle of $T$ whose two other edges must admit $a_0$ as a vertex. These edges are also both incident to the same vertex in the privileged boundary. Hence, they must be interior arcs of the triangulation.
\begin{figure}[b]
\begin{centering}
\includegraphics{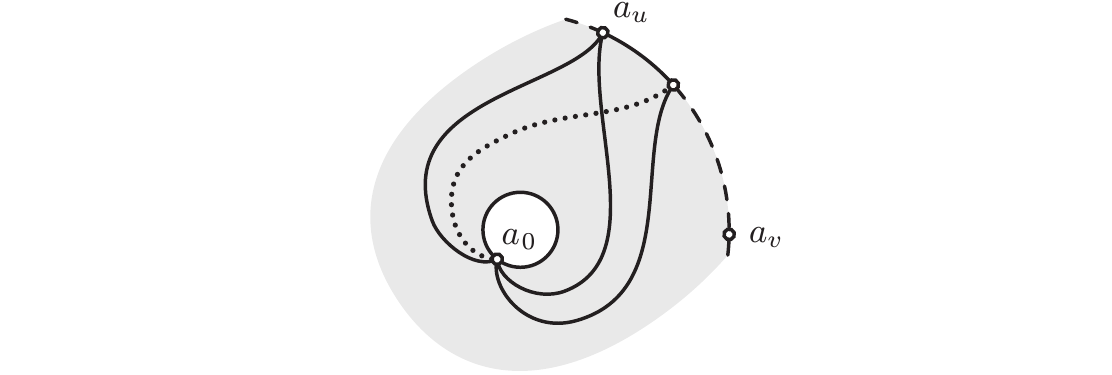}
\caption{The flip used in the proof of Theorem \ref{thm:gammaupper} to transform $U'$ into $V'$ inductively. The introduced edge is dotted.}\label{fig:Lemma1}
\end{centering}
\end{figure}

As such, $n-1$ flips suffice to reach a triangulation with all arcs incident to $a_0$ from either $U$ or $V$. Note that such a triangulation is uniquely determined by the privileged boundary vertex of the triangle incident to $\alpha_0$.

We now perform the above flips within $U$ and $V$ to obtain two triangulations $U'$ and $V'$. Denote by $a_u$ and $a_v$ the privileged boundary vertices of the triangle incident to $\alpha_0$ in respectively $U'$ and $V'$. At most, this necessitates $2n-2$ flips. 

Now to get from $U'$ to $V'$, we proceed as follows. Note that, thinking of the privileged boundary as a graph, the distance of $a_u$ and $a_v$ along this boundary is at most $n/2$. We can perform a flip in $U'$ to obtain a triangulation similar to $U'$, wherein the privileged boundary vertex of the triangle incident to $\alpha_0$ is closer to $a_v$ by $1$ along the privileged boundary (this is illustrated in Fig. \ref{fig:Lemma1}). As such, in at most $n/2$ flips we have transformed $U'$ into $V'$. The result follows.
\end{proof}

It turns out that this straightforward upper bound is (somewhat surprisingly) optimal as will be shown in the sequel. We generalize to an arbitrary number of boundary loops in the next subsection.

\subsection{Upper bounds for surfaces with multiple boundary loops}

The first case we treat is that of marked boundary loops.

\begin{theorem}\label{UBtheorem.47}
Let $\Sigma$ be a surface with $k\geq2$ marked boundary loops. Then there exists a constant $K_k$ which only depends on $k$ such that
$$
\diam(\MFS) \leq \left(4-\frac{2}{k}\right) \, n + K_k.
$$

\end{theorem}

\begin{proof}
We begin by choosing a boundary loop $\alpha_0$ and its vertex which we denote $a_0$. Note that as before, any triangulation has at least two interior arcs incident to $a_0$.

Given two triangulations $U$ and $V$ in $\MFS$ we perform flips within both triangulations until all arcs are incident to $a_0$. This can be done with at most $2n + 8 k -10$ flips for the following reason. A straightforward Euler characteristic argument shows that any triangulation in $\MFS$ has exactly $n+4k-3$ interior arcs. As observed above, at least two of these are already incident to $a_0$, so each triangulation is at most $n+4k-5$ flips away from a triangulation with all arcs incident to $a_0$. We denote the resulting triangulations by $U'$ and $V'$.

Triangulations with the above property are by no means canonical but they do have a very nice structure. Visually, it's useful to think of the vertex $a_0$ as the center of the triangulation. Most arcs (at least when $n$ is considerably bigger than $k$) will be arcs going from a privileged boundary vertex $a_p$ to $a_0$, and will be the unique arc doing so. However, some of them will have a companion arc (or several) also incident to the same two vertices. For this to happen, as they are necessarily non-isotopic arcs, they must enclose some topology. If we consider two successive arcs like this (by successive we mean belonging to the same triangle), they must be boundary arcs of a triangle with a companion loop incident in $a_0$. We shall refer to subsurfaces bounded by such two successive arcs as a {\it pod} and its subsurface bounded by the companion loop as a {\it pea}. A pod is depicted in the right hand side of Fig. \ref{fig:Lemma2}, where the pea is hatched. 
\begin{figure}
\begin{centering}
\includegraphics{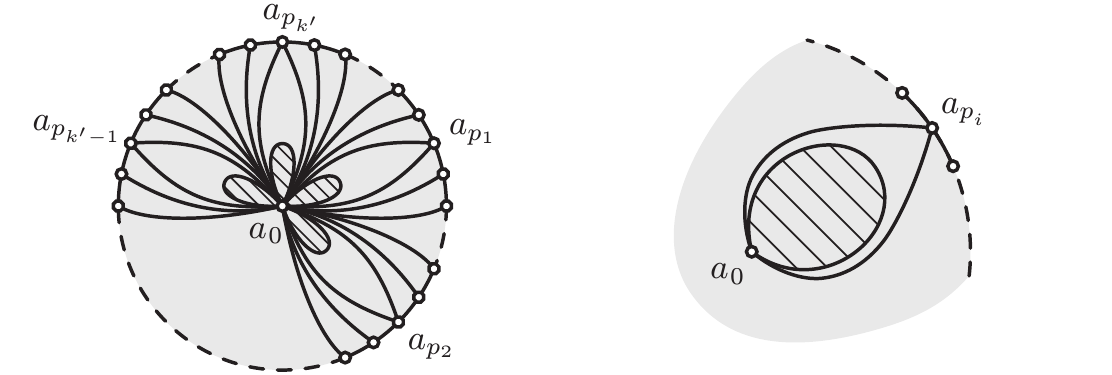}
\caption{Peas in pods.}\label{fig:Lemma2}
\end{centering}
\end{figure}

Observe that that any pea must contain at least one of the $k$ interior boundary loops but could possibly contain several. As such there are at most $k$ peas and as every pod is non-empty, at most $k$ pods. We denote the number of peas and pods by $k'$ and we denote the privileged boundary vertices they correspond to by $a_{p_1},\hdots,a_{p_{k'}}$ where the $p_i \in \Z_n$, $i=1,\hdots,k'$ are ordered along the privileged boundary (clockwise on the left of Fig. \ref{fig:Lemma2}). Note that it is possible that $a_{p_j}=a_{p_{j+1}}$.

Vertices $a_{p_1},\hdots,a_{p_{k'}}$ are separated along the privileged boundary by sequences of vertices (possibly none) which have single arcs to $a_0$ (see the left hand side of Fig. \ref{fig:Lemma2}). We call these sequences {\it gaps}. For both $U'$ and $V'$ we want to find the largest gap. As there are $n$ vertices on the boundary separated by at most $k'$ pods, there is always a gap of size at least $\frac{n}{k'}\geq \frac{n}{k}$, i.e., with the notation used for a generic such triangulation above, an $i_0$ with 
$$d_{\Z_n}(p_{i_0} ,p_{i_0 +1}) \geq \frac{n}{k}.$$
We now consider the largest gaps in both $U'$ and $V'$. The set of vertices not found in the gaps are both of cardinality at most $n - \frac{n}{k}$. We distinguish two cases.

{\underline{Case 1}:} Some vertex $a_{g}$ does not belong to either the gap of $U'$ or the gap of $V'$.

The strategy here is to flip $U'$ and $V'$ into triangulations with a single pod at $a_g$. They will thus coincide outside of the pod and it will suffice to flip inside the pod a number of times depending only on $k$ to relate the two triangulations. 

We begin by observing that a pod can be moved to neighboring vertex by a single flip unless another pod obstructs its passage (see the left hand side of Fig. \ref{fig:Lemma3}). 

For both triangulations we proceed in the same way. We ``condemn" the gap, and flip the pods until they reach $a_g$ without passing through the condemned gap as follows. We take one of the pods bounding the gap and flip it until it reaches another pod or vertex $a_g$. In the former case, the two pods are transformed into a single pod by the flip portrayed on the right of Fig. \ref{fig:Lemma3}. 
\begin{figure}
\begin{centering}
\includegraphics{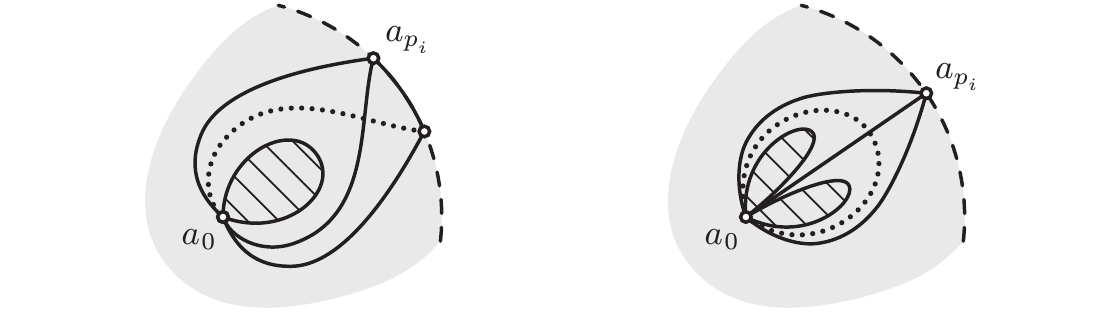}
\caption{A flip that moves a pod (left) and joins two pods (right). In each case, the introduced edge in dotted.}\label{fig:Lemma3}
\end{centering}
\end{figure}

We then continue to flip until reaching another pod (or $a_g$) etc. In any case, once vertex $a_g$ has been reached, the process stops. At this point there are no pods left between $a_g$ and the condemned gap on one side. We do the same on the other side. 

We now count at most how many flips were necessary. As there were originally at most $k$ pods, at most $k-1$ flips we necessary to join pods. All of the other flips have reduced by $1$ the distance between the pods bounding the condemned gap, thus there were at most $n-\frac{n}{k}$ such flips.

As we performed this on both triangulations, the total number of flips that have been carried out does not exceed
$$
(2-\frac{2}{k})\, n + 2 k -2
$$
If we denote $U''$ and $V''$ the resulting triangulations, we now have two triangulations that differ only on a single pea which contains all of the topology and where all arcs are incident to $a_0$. We now flip inside the pea. As a subsurface, it is homeomorphic to $\Sigma_1$, thus 
$$
d(U'',V'') \leq \diam(  \MF(\Sigma_1))
$$
and this diameter is equal to some constant $K'_k$ which only depends on $k$. Using these estimates and our original estimates on the distances to $U'$ and $V'$ we obtain
$$
d(U,V) \leq (2-\frac{2}{k})\, n  + K'_k + 2 k - 2 + 2n + 8k- 10
$$
thus setting $K_k:= {K'}_k + 10 k - 12$ we obtain
$$
d(U,V) \leq (4-\frac{2}{k})\, n  + K_k
$$
as desired. 

{\underline{Case 2}:} Each of the vertices of the privileged boundary belongs to the gap of $U'$ or to the gap of $V'$.

This is the easier case, as now all the privileged boundary vertices incident to pods of $U'$ lie in a sector disjoint from another section containing all the privileged boundary vertices incident to pods of $V'$. We condemn the two gaps and move the pods by flips as done above: choose a pod in $U'$ at the boundary of the gap and flip it towards the other boundary, in the direction that keeps it outside the gap. This proceeds until $U'$ is transformed into a triangulation with a single pod at the other boundary of the gap. Call $a_l$ the vertex in the privileged boundary that is incident to the remaining pod. We now flip $V'$ similarly but in the opposite direction (in order to keep the pods outside the condemned gap, one just need to start the flipping process from the appropriate boundary). We continue to flip the resulting triangulation until it has a single pod in $a_l$. We denote the resulting triangulations by $U''$ and $V''$. Note that, as above, there were at most $2k-2$ flips that served to join adjacent pods. All other flips brought the outermost pods one closer to $a_l$. Hence, there were at most $n-1$ such flips. Thus in total
$$
d(U,U'')+d(V,V'') \leq n - 1 +2k -2.
$$
Now $U''$ and $V''$ differ in a single pea, and thus as above satisfy
$$
d(U'',V'') \leq \diam(  \MF(\Sigma_1)).
$$
We can conclude that, taking the same constant $K_k$ as previously that
$$
d(U,V) \leq 3 n  + K_k \leq (4-\frac{2}{k}) n + K_k\mbox{.}
$$
Note that the second inequality holds because $k\geq 2$.
\end{proof}

Observe that this implies an upper bound on the order of $3n$ when $k=2$ that is, when $\Sigma=\Pi$. An adaptation of the above proof for unmarked boundary loops gives stronger upper bounds. In particular the following is true.

\begin{theorem}\label{UBtheorem.48}
Let $\Sigma$ be a disk with $k$ unmarked boundary loops. Then there exists a constant $K_k$ which only depends on $k$ such that
$$
\diam(\MFS) \leq \left(3-\frac{1}{2k}\right) \, n + K_k.
$$
\end{theorem}

\begin{proof}
Let $U$ and $V$ be triangulations in $\MFS$. We begin by observing the following: every boundary loop is {\it close} to some vertex on the privileged boundary.

More precisely, consider the dual of $U$, i.e. the graph $D$ whose vertices are the triangles of $U$ and whose edges connect two triangles that share an edge. Observe that $D$ is connected. Let $t$ be the triangle of $U$ incident to some boundary loop. Consider a triangle $t'$ of $U$ incident to the privileged boundary that is closest to $t$ in $D$. The distance in $D$ between $t$ and $t'$ only depends on $k$. Indeed, consider a geodesic between $t$ and $t'$ in $D$. The only triangle incident to the privileged boundary along this geodesic is $t'$. Hence the length of this geodesic cannot depend on $n$, but only on $k$. The vertex of $t'$ on the privileged boundary is the one we call \emph{close} to the boundary loop. Now observe that flipping the arcs of $U$ dual to the edges of our geodesic from $t'$ to $t$ will introduce a triangle incident to both the boundary loop and the privileged boundary vertex it is close to. We then say that the boundary loop is {\it hanging off} this vertex.

We carry out the above sequence of flips for every boundary loop. Note that these flips never remove an arc incident to the privileged boundary. Hence, once a boundary loop is hanging off a privileged boundary vertex, it will be left so by the later flips. The number of flips needed to transform both $U$ and $V$ as described above does not depend on $n$, but only on $k$. We denote the resulting triangulations by $U'$ and $V'$.

By construction, all the boundary loops of $U'$ and $V'$ hang off of privileged boundary vertices, either by itself or in a bunch as depicted in Fig. \ref{fig:ThmUM1}. Observe that if two boundary curves hang off the same vertex, then they are separated by at least one other triangle.
\begin{figure}
\begin{centering}
\includegraphics{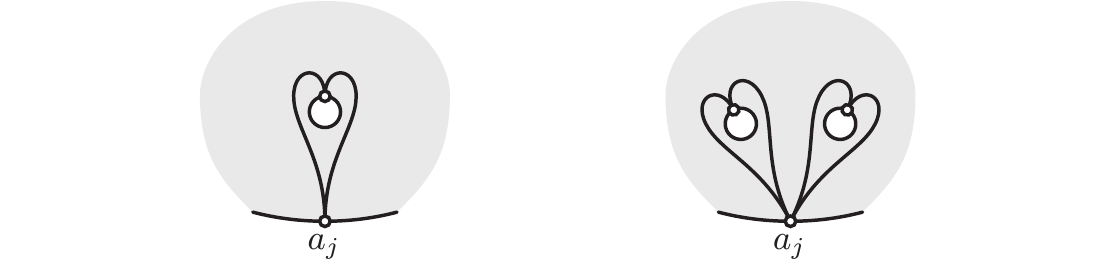}
\caption{Boundary curves ``hanging off" privileged boundary vertices}
\label{fig:ThmUM1}
\end{centering} 
\end{figure}
\begin{figure}
\begin{centering}
\includegraphics{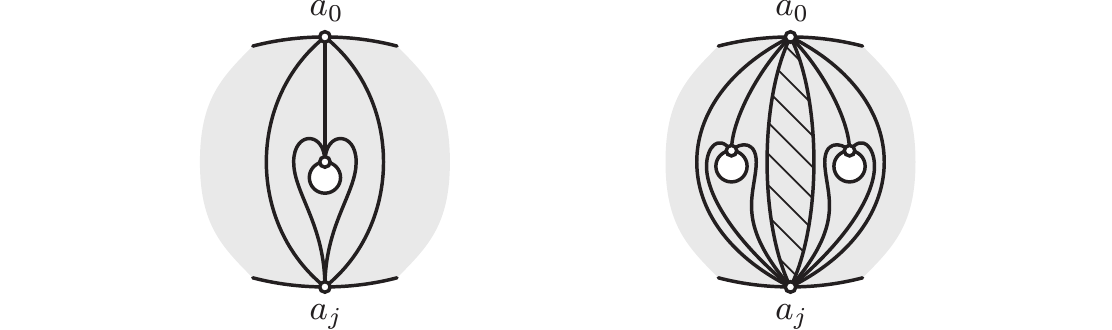}
\caption{A pod with a unique boundary loop (left), and a pod with several boundary loops (right).}
\label{fig:ThmUM2}
\end{centering}
\end{figure}

For a moment we forget all of the triangles of $U'$ and $V'$ that are not incident to a boundary loop. We consider the collection of privileged boundary vertices that have boundary loops hanging off of them in either $U'$ or $V'$. There are at most $2k$ such vertices and, as in the previous proof we consider the gaps of successive privileged boundary vertices without anything hanging off of them. We now consider the largest gap, whose size is at least $\frac{n}{2k}$.

We choose one of the privileged boundary vertices {\it contained} in the gap and denote it $a_0$. We carry out flips within both $U'$ and $V'$ to increase the interior degree of $a_0$ but (and this is important) without flipping the edges of any triangle incident to a boundary loop. Once this is done, all other arcs are incident in $a_0$. The vertices in the boundary loops are incident to a unique arc which joins them to $a_0$ as shown in Fig. \ref{fig:ThmUM2}.

The two triangulations look very similar with the exception of the placement of the boundary loops. They are all found in sectors (which we call {\it pods}) bounded by two arcs between $a_0$ and some other privileged boundary vertex $a_j$, possibly by themselves, possibly with other boundary loops (see Fig. \ref{fig:ThmUM2}). As in the previous theorem, we want to put these boundary loops in {\it peas} so that they are easy to move, but this time we use the privileged boundary vertex $a_0$ as a base for all the peas. 

To do this, we perform flips inside each pod so that all the boundary loops inside a given pod become enclosed in a single pea attached to $a_0$. This may take a certain number of flips but an upper bound on how many is given by 
$$
\diam( \MF({\Sigma'}_2))\mbox{,}
$$
where $\Sigma'$ is the surface inside the pod. As $\Sigma'$ has at most $k$ interior boundary curves, this is bounded by some function of $k$. Note that each boundary curve is inside some pea belonging to a pod attached to both $a_0$ and some other privileged boundary vertex $a_j$. This $a_j$ is of course the original vertex that the boundary curve was {\it close} to. 

We can now begin to move the pods around. The idea is to move the pods clockwise around $a_0$ using the flip depicted on the left of Fig. \ref{fig:ThmUM4}.


\begin{figure}
\begin{centering}
\includegraphics{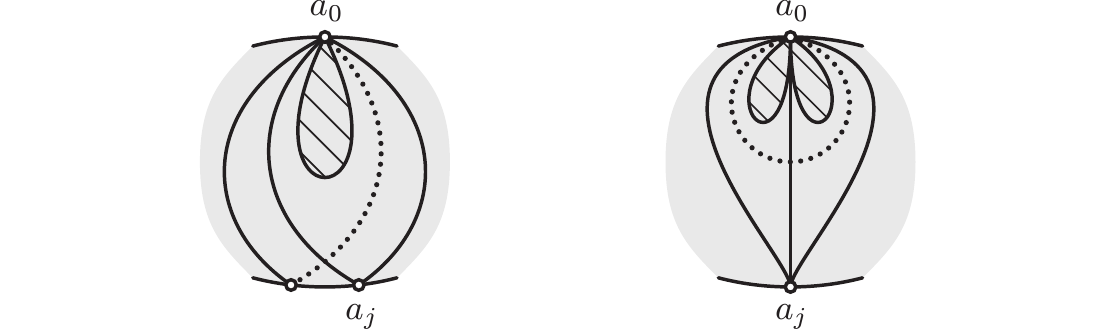}
\caption{A flip that moves a pod by one vertex clockwise around $a_0$ (left) and a flip that joins two pods (right). In each case, the introduced edge is dotted.}
\label{fig:ThmUM4}
\end{centering}
\end{figure}

We will refer to the number of boundary loops in a pea or in a pod as the pea or the pod's {\it multiplicity}. We begin as follows: we consider the first pod clockwise around $a_0$ in either triangulations. If both triangulations have such a pod we choose the one with the largest multiplicity. If they both have a pod of the same multiplicity we leave them as they are and look for the next pod clockwise in either triangulation. The selected pod is incident to $a_0$ and to another privileged boundary vertex $a_j$.

If one of the triangulations has no pod at incident to $a_j$, we move the pod clockwise in the one that does to the next vertex incident to a pod on either triangulation. As in the previous theorem, moving a pod by one vertex requires one flip as shown on the left of Fig. \ref{fig:ThmUM4}. 

If however both triangulations have pods with different multiplicities incident to $a_j$, we first perform flips inside the one with the larger multiplicity to split it into two pods, each containing a pea attached to $a_0$. We make the first pod (in the direction of our orientation) with the same multiplicity as the pod of the other triangulation and the second with whatever multiplicity comes from the leftover boundary loops. Again, this splitting operation requires a number of flips but no more than 
$$
\diam( \MF({\Sigma'}_2))\mbox{,}
$$
where $\Sigma'$ is the surface inside the pod, as above. We then move this second pod by flips to the next vertex clockwise with a pod on either triangulation. Whenever the moving pod encounters another pod in its own triangulation, we perform a single flip to join them as shown on the right of Fig. \ref{fig:ThmUM4}, and we iterate the process until we reach the last pod clockwise around $a_0$.

The two resulting triangulations have pods of the same multiplicity incident to the same privileged boundary vertices. More precisely, these triangulations only possibly differ in the way the peas are triangulated. We therefore finally perform flips inside the peas in order to make the two triangulations coincide. Note that the number of these flips does not depend on $n$ but only on $k$.

Let us now take a look at how many flips we have performed. 

We began by tweaking both triangulations so that all boundary loops hung off privileged boundary vertices. This required a number of flips that does not depend on $n$, but only on $k$, which we call $K'_k$. We then increased the interior degree of $a_0$. By an Euler characteristic argument, this required at most $2n+4k-6$ flips. Moving pods from one end of the gap to the other required at most $n$ flips to which the size of the gap must be subtracted, thus at most $n - \frac{n}{2k}$ flips. 

In several places we had to transform two triangulations in $\MF({\Sigma'}_2)$ into one another for some subsurface $\Sigma'$ of $\Sigma$. The number of flips needed to perform every such transformation in any possible subsurface $\Sigma'$ is bounded above by a number $K''_k$ that does not depend on $n$. We had to do these transformations at most $k$ times to attach the peas to $a_0$, and once every time a pod had to be split. The splitting operation was performed at most $2k$ times because the number of pods in the two triangulations is bounded above by $2k$. Hence the total number of flips performed to modify triangulations in $\MF({\Sigma'}_2)$ is at most $3kK''_k$.

Likewise, we may have had to join pods together requiring in total at most $2k$ flips. The final flipping inside the peas was bounded above by a number $K'''_k$ that does not depend on $n$. 

Setting $K_k:= K'_k+3kK''_k+K'''_k+6k-6$, we obtain an upper bound of
$$
\left( 3 - \frac{1}{2k}\right) n + K_k\mbox{,}
$$
on the diameter of $\MF(\Sigma_n)$ as desired.
\end{proof}

\subsection{A few other cases}

The proof of Theorem \ref{UBtheorem.47} still works when some of the boundary loops are replaced by interior points. The only difference is that some of the peas will enclose interior points instead of boundary loops. Hence:

\begin{theorem}
Let $\Sigma$ be a surface with $l$ marked boundary loops and $k$ marked interior vertices. If $k+l$ is not less than $2$, then there exists a constant $K_{k+l}$ which only depends on $k+l$ such that
$$
\diam(\MFS) \leq \left(4-\frac{2}{k+l}\right) \, n + K_{k+l}\mbox{.}
$$
\end{theorem}

Adapting the proof of Theorem \ref{UBtheorem.48} to surfaces with interior points and boundary loops is not immediate. Indeed, a point and a boundary loop cannot be exchanged. However, if all the boundary loops are replaced by interior vertices, a straightforward adaptation of this proof will work. As above, the only difference is that peas will enclose vertices instead of boundary loops so we only give the main steps.

\begin{theorem}
Let $\Sigma$ be a surface with $k$ unmarked interior vertices. If $k$ is not less than $2$, then there exists a constant $K_k$ which only depends on $k$ such that
$$
\diam(\MFS) \leq \left(3-\frac{1}{2(k-1)}\right) \, n + K_{k}\mbox{.}
$$
\end{theorem}
\begin{proof}
Given any two triangulations $U$ and $V$, we begin by choosing any interior vertex and perform flips to increase its incidence in both triangulations. This requires $2n$ flips in total plus a constant which only depends on $k$. The resulting triangulations now have peas in pods where the peas have the form of a loop surrounding a single arc between two interior vertices. 

As in the proof of Theorem \ref{UBtheorem.48}, we consider the largest gap between two pods and move them around in an almost identical fashion. The gap is of size at least $n/(2k-2)$ as we have already used one of the interior vertices as the ``center" of the triangulation. The other details are identical to Theorem \ref{UBtheorem.48} and we leave them to the dedicated reader.
\end{proof}

\section{Lower bounds for $\Gamma$}\label{Asection.4}

In this section, we prove the following lower bound on the diameter of $\mathcal{MF}(\Gamma_n)$:
\begin{equation}\label{Aequation.4.0}
\diam(\mathcal{MF}(\Gamma_n))\geq \lfloor{ \frac{5}{2} n }\rfloor-2.
\end{equation}

This will be done exhibiting two triangulations $A_n^-$ and $A_n^+$ in $\MF(\Gamma_n)$ with the distance the right-hand side of (\ref{Aequation.4.0}). These triangulations are built by modifying the triangulation $Z_n$ of $\Delta_n$ depicted in Fig. \ref{Afigure.4.1}, where $\Delta_n$ is a disc with $n$ marked vertices on the boundary.

The interior arcs of $Z_n$ form a zigzag, i.e. a simple path that alternates between left and right turns. This path starts at vertex $a_n$, ends at vertex $a_{n/2}$ when $n$ is even, and at vertex $a_{\lceil{n/2\rceil+1}}$ when $n$ is odd. When $n$ is greater than $3$, triangulation $Z_n$ has an ear in $a_1$ and another ear in $a_{\lfloor{n/2}\rfloor+1}$. When $n$ is equal to $3$, this triangulation is made up of a single triangle which is an ear in all three vertices. Observe that $Z_n$ cannot be defined when $n$ is less than $3$.

Assume that $n\geq3$.
\begin{figure}[b]
\begin{centering}
\includegraphics{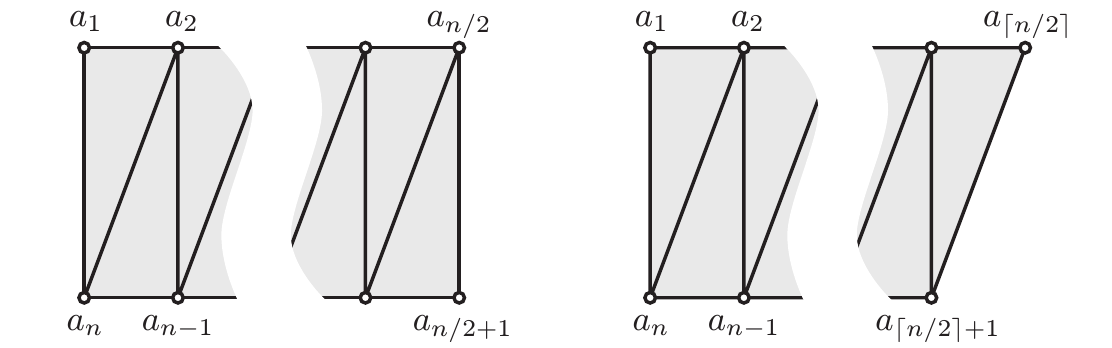}
\caption{The triangulation $Z_n$ of $\Delta_n$ depicted when $n$ is even (left) and odd (right).}\label{Afigure.4.1}
\end{centering}
\end{figure}
A triangulation $A_n^-$ of $\Gamma_n$ can be built by ``piercing the ear" of $Z_n$ in $a_1$: formally, we place a boundary loop $\alpha_0$ with a vertex $a_0$ inside the ear and re-triangulate the pierced ear as shown on the top of Fig. \ref{Afigure.4.2}. Another triangulation $A_n^+$ of $\Gamma_n$ can be built by piercing  the ear of $Z_n$ in $a_{\lfloor{n/2}\rfloor+1}$, by placing vertex $a_0$ on the boundary of the resulting hole, and by re-triangulating the pierced ear as shown in the bottom of Fig. \ref{Afigure.4.2}.

In the following triangulations $A_n^-$ and $A_n^+$ are understood as elements of $\MF(\Gamma_n)$ that is, up to homeomorphism.
\begin{figure}
\begin{centering}
\includegraphics{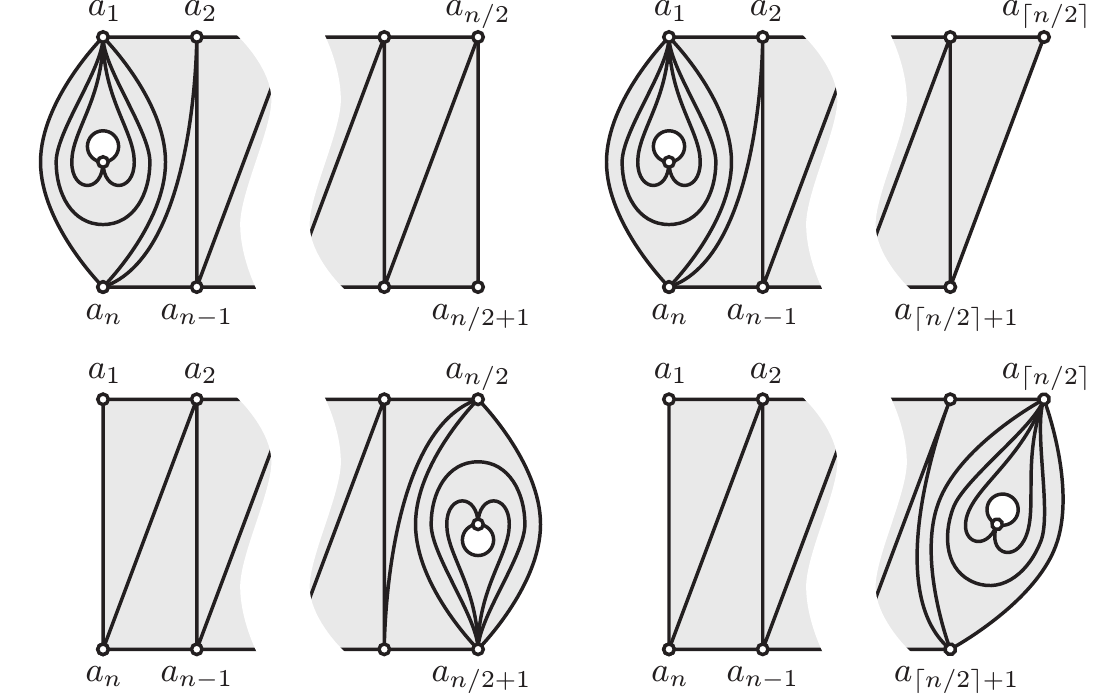}
\caption{The triangulations $A_n^-$ (top row) and $A_n^+$ (bottom row) of $\Gamma_n$ depicted when $n$ is even (left) and odd (right). For simplicity, vertex $a_{0}$ is unlabeled here.}\label{Afigure.4.2}
\end{centering}
\end{figure}

We will also define them when $1\leq n\leq 2$. Triangulations $A_2^-$ and $A_2^+$ are the triangulations in $\MF(\Gamma_2)$ that contain a loop arc at respectively vertex $a_1$ and vertex $a_2$, as shown on Fig. \ref{Afigure.2.1}. Triangulations $A_1^-$ and $A_1^+$ will both be equal to the unique triangulation in $\MF(\Gamma_1)$, also shown on Fig. \ref{Afigure.2.1}.

One of the main steps in our estimates will be to show, for every integer $n$ greater than $2$, the following inequality:
\begin{equation}\label{Aequation.4.1}
d(A_n^-,A_n^+)\geq\min(\{d(A_{n-1}^-,A_{n-1}^+)+3,d(A_{n-2}^-,A_{n-2}^+)+5\})\mbox{.}
\end{equation}

This inequality will be obtained using well chosen vertex deletions or sequences of them. For instance, for $n\geq2$ , observe that deleting vertex $a_n$ from both $A_n^-$ and $A_n^+$ results in triangulations isomorphic to $A_{n-1}^-$ and $A_{n-1}^+$. More precisely, once the vertex has been deleted, the vertices need be relabeled in order to obtain $A_{n-1}^-$ and $A_{n-1}^+$. The natural way to do this is to shift the labels of all subsequent vertices to the deleted vertex as:
$$
a_i \to a_{i-1}\mbox{.}
$$

This relabeling provides a map onto the triangulations of $\Gamma_{n-1}$. For future reference we call any such map a {\it vertex relabeling}. We can now precisely state the observation we need: the triangulations $A_{n}^-\contract{n}$, resp. $A_{n}^+\contract{n}$ are isomorphic to $A_{n-1}^-$, resp. $A_{n-1}^+$ via the same vertex relabeling. This can be checked using Fig. \ref{Afigure.2.1} when $2\leq{n}\leq4$ and Fig. \ref{Afigure.4.2} when $n\geq3$.

According to Theorem \ref{Atheorem.2.3}, it follows from this observation that if there exists a geodesic between $A_n^-$ and $A_n^+$ with at least $3$ flips incident to $\alpha_n$, then
\begin{equation}\label{Aequation.4.2}
d(A_n^-,A_n^+)\geq{d(A_{n-1}^-,A_{n-1}^+)+3}\mbox{,}
\end{equation}
and inequality (\ref{Aequation.4.1}) holds in this case. Now assume that $n\geq3$ and observe that for any integer $i$ so that $1\leq{i}<n$ and any $j\in\{n-i,n-i+1\}$, deleting vertices $a_i$ and $a_j$ from $A_n^-$ and from $A_n^+$ results in triangulations of $\Gamma_n$ isomorphic to $A_{n-2}^-$ and $A_{n-2}^+$ respectively. The isomorphism between these triangulations comes from the same vertex relabeling as above. Hence, if there exists a geodesic between $A_n^-$ and $A_n^+$ with at least $3$ flips incident to $\alpha_i$, and a geodesic between $A_n^-\contract{i}$ and $A_n^+\contract{i}$ with at least $2$ flips incident to $\alpha_j$, then invoking Theorem \ref{Atheorem.2.3} twice yields
\begin{equation}\label{Aequation.4.3}
d(A_n^-,A_n^+)\geq{d(A_{n-2}^-,A_{n-2}^+)+5}\mbox{,}
\end{equation}
and inequality (\ref{Aequation.4.1}) also holds in this case. Observe that (\ref{Aequation.4.2}) and (\ref{Aequation.4.3}) follow from the existence of particular geodesic paths. The rest of the section is devoted to proving the existence of geodesic paths that imply at least one of these inequalities.

Observe that $\alpha_n$ is not incident to the same triangle in $A_n^-$ and in $A_n^+$. Therefore, at least one flip is incident to $\alpha_n$ along any geodesic from $A_n^-$ to $A_n^+$. The proof will consist in studying these geodesics depending on the arc introduced by their first flip incident to $\alpha_n$, which is the purpose of the next three lemmas.

\begin{lemma}\label{Alemma.4.1}
Let $n$ be an integer greater than $2$. Consider a geodesic from $A_n^-$ to $A_n^+$ whose first flip incident to arc $\alpha_n$ introduces an arc with vertices $a_0$ and $a_n$. If $\alpha_n$ is incident to at most $2$ flips along this geodesic then $\alpha_1$ is incident to at least $3$ flips along it.
\end{lemma}
\begin{proof}
Let $(T_i)_{0\leq{i}\leq{k}}$ be a geodesic from $A_n^-$ to $A_n^+$. Assume that the first flip incident to $\alpha_n$ along $(T_i)_{0\leq{i}\leq{k}}$ is the $j$-th one, and that it introduces an arc with vertices $a_0$ and $a_n$. This flip must then be the one shown on the left of Fig. \ref{Afigure.4.3}.

Assume that at most one flip along $(T_i)_{0\leq{i}\leq{k}}$ other than the $j$-th one is incident to $\alpha_n$. In this case, there must be exactly one such flip among the last $k-j$ flips of $(T_i)_{0\leq{i}\leq{k}}$, say the $l$-th one. Moreover, this flip replaces the triangle of $T_j$ incident to $\alpha_n$ by the triangle of $A_n^+$ incident to $\alpha_n$. There is only one way to do so, depicted in the right of Fig. \ref{Afigure.4.3}. It can be seen that this flip is incident to $\alpha_1$.

The rest of the proof consists in an indirect argument. Assume that at most one flip along $(T_i)_{0\leq{i}\leq{k}}$ other than the $l$-th one is incident to $\alpha_1$. In this case, the first flip incident to $\alpha_1$ along $(T_i)_{0\leq{i}\leq{k}}$, say the $l'$-th one, replaces the triangle of $A_n^-$ incident to arc $\alpha_1$ by the triangle of $T_{l-1}$ incident to this arc.
\begin{figure}
\begin{centering}
\includegraphics{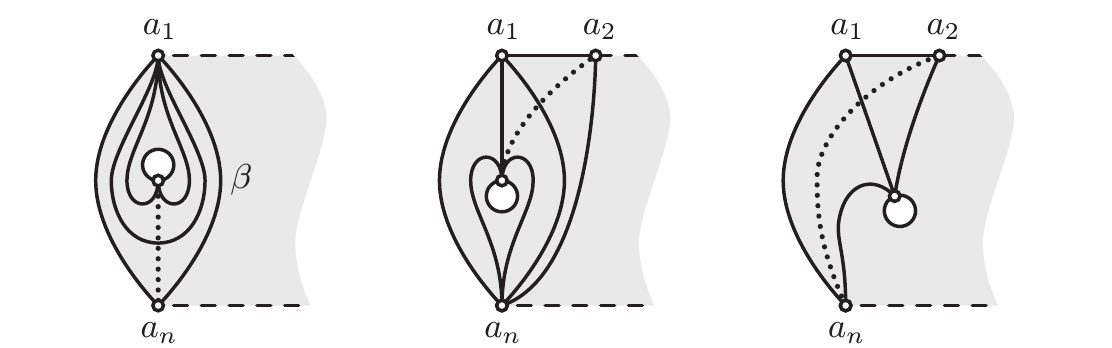}
\caption{The $i$-th flip performed along path $(T_i)_{0\leq{i}\leq{k}}$ in the proof of Lemma \ref{Alemma.4.1}, where $i=j$ (left), $i=l'$ (center), and $i=l$ (right) . In each case, the introduced edge is dotted and the solid edges belong to triangulation $T_{i-1}$.}\label{Afigure.4.3}
\end{centering}
\end{figure}
There is only one way to do so, depicted in the center of Fig. \ref{Afigure.4.3}. One can see that the triangle of $T_{l'-1}$ incident to arc $\alpha_0$ cannot be identical to the triangle of $A_n^-$ incident to this arc. Hence one of the first $l'-1$ flips along $(T_i)_{0\leq{i}\leq{k}}$, say the $j'$-th one, removes the triangle of $A_n^-$ incident to $\alpha_0$.

As $j'<l'$, arc $\beta$ shown in the left of Fig. \ref{Afigure.4.3} belongs to both $T_{j'-1}$ and $T_{j'}$. The portion of each of these triangulations bounded by arcs $\alpha_n$ and $\beta$ belongs to $\MF(\Gamma_2)$. According to Proposition \ref{Aproposition.2.1}, the $j'$-th flip along $(T_i)_{0\leq{i}\leq{k}}$ is then incident to $\alpha_n$. As the $j$-th and $l$-th flips along this path are also incident to $\alpha_n$, this contradicts the assumption that $\alpha_n$ is incident to at most $2$ flips along $(T_i)_{0\leq{i}\leq{k}}$. Therefore $\alpha_1$ must be incident to at least three flips along this geodesic.
\end{proof}

\begin{lemma}\label{Alemma.4.2}
Let $n$ be an integer greater than $2$. Consider a geodesic from $A_n^-$ to $A_n^+$ whose first flip incident to $\alpha_n$ introduces an arc with vertices $a_1$ and $a_2$. If $\alpha_n$ is incident to at most $2$ flips along this geodesic then $\alpha_1$ is incident to at least $4$ flips along it.
\end{lemma}
\begin{proof}
Let $(T_i)_{0\leq{i}\leq{k}}$ be a geodesic from $A_n^-$ to $A_n^+$ whose first flip incident to $\alpha_n$, say the $j$-th one, introduces an arc with vertices $a_1$ and $a_2$. This flip must then be the one shown on the left of Fig. \ref{Afigure.4.4}. Note that it is incident to arc $\alpha_1$.

Assume that at most one flip along $(T_i)_{0\leq{i}\leq{k}}$ other than the $j$-th one is incident to $\alpha_n$. In this case, there must be exactly one such flip among the last $k-j$ flips of $(T_i)_{0\leq{i}\leq{k}}$, say the $l$-th one. Moreover, this flip replaces the triangle of $T_j$ incident to $\alpha_n$ by the triangle of $A_n^+$ incident to $\alpha_n$. There is only one way to do so, depicted in the right of Fig. \ref{Afigure.4.4}. Note that this flip is also incident to $\alpha_1$.

Finally, as the arc introduced by the $j$-th flip along $(T_i)_{0\leq{i}\leq{k}}$ is not removed before the $l$-th flip, there must be two more flips incident to arc $\alpha_1$ along this geodesic: the flip that removes the loop arc with vertex $a_1$ shown on the left of Fig. \ref{Afigure.4.4} and the flip that introduces the loop arc with vertex $a_2$ shown on the right of the figure. This proves that at least four flips are incident to $\alpha_1$ along $(T_i)_{0\leq{i}\leq{k}}$.
\end{proof}

\begin{figure}
\begin{centering}
\includegraphics{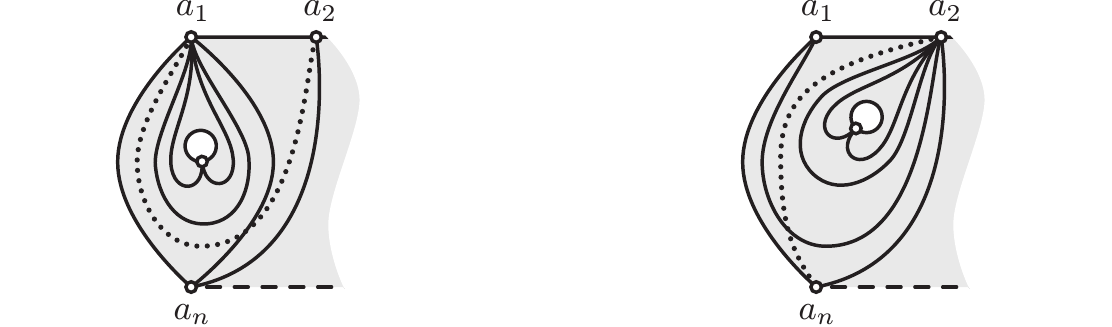}
\caption{The $i$-th flip performed along path $(T_i)_{0\leq{i}\leq{k}}$ in the proof of Lemma \ref{Alemma.4.2}, where $i=j$ (left), and $i=l$ (right). In each case, the introduced edge is dotted and the solid edges belong to triangulation $T_{i-1}$.}\label{Afigure.4.4}
\end{centering}
\end{figure}

\begin{lemma}\label{Alemma.4.3}
For $n\geq4$, consider a geodesic from $A_n^-$ to $A_n^+$ whose first flip incident to $\alpha_n$ introduces an arc with vertices $a_1$ and $a_p$, where $2<p<n$. Then:
$$
d(A_n^-,A_n^+)\geq\min(\{d(A_{n-1}^-,A_{n-1}^+)+3,d(A_{n-2}^-,A_{n-2}^+)+5\})\mbox{.}
$$
\end{lemma}
\begin{proof}
Let $(T_i)_{0\leq{i}\leq{k}}$ be a geodesic from $A_n^-$ to $A_n^+$ whose first flip incident to $\alpha_n$, say the $j$-th one, introduces an arc with vertices $a_1$ and $a_p$, where $2<p<n$. This flip is depicted in Fig. \ref{Afigure.4.5}. It is first shown that $T_j$ has an ear in some vertex $a_q$ where $2\leq{q}<p$ if $p\leq{\lceil{n/2}\rceil}$, and $p<q\leq{n}$ otherwise.

Assume that $p$ is not greater than $\lceil{n/2}\rceil$. Consider the arc of $T_j$ with vertices $a_1$ and $a_p$ shown as a solid line on the left of Fig. \ref{Afigure.4.5}.
\begin{figure}
\begin{centering}
\includegraphics{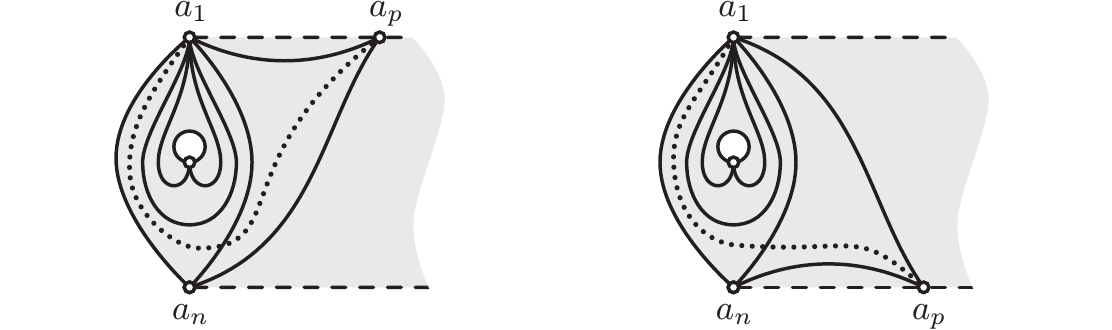}
\caption{The $j$-th flip performed along path $(T_i)_{0\leq{i}\leq{k}}$ in the proof of Lemma \ref{Alemma.4.3} when $p\leq{\lceil{n/2}\rceil}$ (left) and when $p>\lceil{n/2}\rceil$ (right). In each case, the introduced edge is dotted and the solid edges belong to triangulation $T_{j-1}$.}\label{Afigure.4.5}
\end{centering}
\end{figure}
The portion of $T_j$ bounded by this arc and by arcs $\alpha_1$, ..., $\alpha_{p-1}$ is a triangulation of disc $\Delta_p$. If $p>3$, then this triangulation has at least two ears, and one of them is also an ear of $T_j$ in vertex $a_q$ where $2\leq{q}<p$. If $p=3$ this property still necessarily holds with $q=2$ since the triangulation of $\Delta_p$ induced by $T_j$ is made up of a single triangle. 

Now assume that $\lceil{n/2}\rceil<p<n$. Consider the arc with vertices $a_1$ and $a_p$ introduced by the $j$-th flip along $(T_i)_{0\leq{i}\leq{k}}$ and shown as a dotted line on the right of Fig. \ref{Afigure.4.5}. The portion of $T_j$ bounded by this arc and by arcs $\alpha_p$, ..., $\alpha_n$ is a triangulation of $\Delta_{n-p+2}$. Since $n-p+2$ is at least $3$, an argument similar to the one used in the last paragraph shows that $T_j$ has an ear in some vertex $a_q$ where $p<q\leq{n}$.

This proves that $T_j$ has an ear in $a_q$ so that:
$$
2\leq{q}<p\leq\lceil{n/2}\rceil\mbox{ or }\lceil{n/2}\rceil<p<q\leq{n}\mbox{.}
$$

Note that, if one cuts geodesic $(T_i)_{0\leq{i}\leq{k}}$ at triangulation $T_j$, then Lemma \ref{Alemma.2.75} can be invoked for each of the resulting portions. Doing so, we find that either $\alpha_{q-1}$ and $\alpha_q$ are both incident to exactly $3$ flips along this geodesic or one of these arcs is incident to at least $4$ flips along it. We review the two cases separately.

If $\alpha_r$ is incident to at least $4$ flips along $(T_i)_{0\leq{i}\leq{k}}$, where $r$ is equal to $q-1$ or to $q$, then Theorem \ref{Atheorem.2.3} yields
\begin{equation}\label{Alemma.4.3.equation.1}
d(A_n^-,A_n^+)\geq{d(A_n^-\contract{r},A_n^+\contract{r})+4}\mbox{.}
\end{equation}

Call $s=n-q+1$ and observe that arc $\alpha_s$ is not incident to the same triangle in $A_n^-\contract{r}$ and in $A_n^+\contract{r}$. Hence, some flip must be incident to this arc along any geodesic between $A_n^-\contract{r}$ and $A_n^+\contract{r}$. Invoking Theorem \ref{Atheorem.2.3} again, we find
\begin{equation}\label{Alemma.4.3.equation.2}
d(A_n^-\contract{r},A_n^+\contract{r})\geq{d(A_n^-\contract{r}\contract{s},A_n^+\contract{r}\contract{s})+1}\mbox{.}
\end{equation}

As $A_n^-\contract{r}\contract{s}$ and $A_n^+\contract{r}\contract{s}$ are isomorphic to $A_{n-2}^-$ and $A_{n-2}^+$ by the same vertex relabeling, the desired result is obtained combining (\ref{Alemma.4.3.equation.1}) and (\ref{Alemma.4.3.equation.2}).

It is assumed in the remainder of the proof that $\alpha_{q-1}$ and $\alpha_q$ are both incident to exactly $3$ flips along $(T_i)_{0\leq{i}\leq{k}}$. If $q=n$, then the result immediately follows from Theorem \ref{Atheorem.2.3} because $A_n^-\contract{n}$ and $A_n^+\contract{n}$ are isomorphic to $A_{n-1}^-$ and $A_{n-1}^+$ by the same vertex relabeling. We will therefore also assume that $q<n$.

Call $r=q-1$ if $p\leq\lceil{n/2}\rceil$ and $r=q$ otherwise.
\begin{figure}
\begin{centering}
\includegraphics{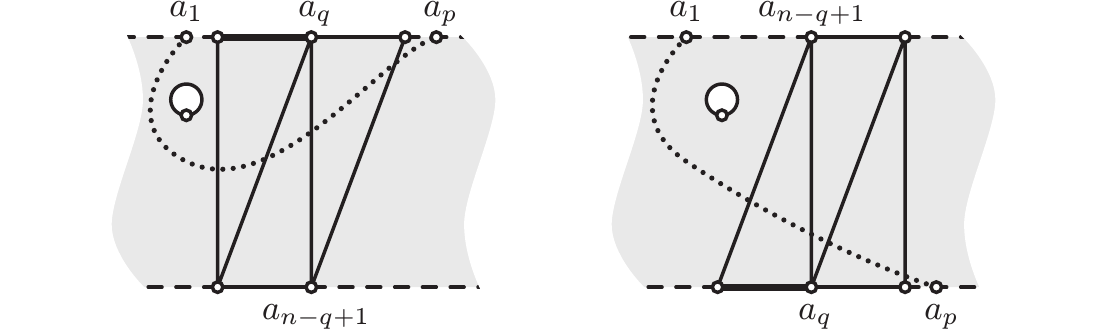}
\caption{The portion of either $A_n^-$ or $A_n^+$ next to vertex $a_q$ when $2\leq{q}\leq\lceil{n/2}\rceil$ (left) and $\lceil{n/2}\rceil<q<n$ (right). The dotted line shows an arc with vertices $a_1$ and $a_p$ used in the proof of Theorem \ref{Atheorem.4.1}, and the bold line shows the boundary arc removed when vertex $a_r$ is deleted.}\label{Afigure.4.6}
\end{centering}
\end{figure}
Cutting geodesic $(T_i)_{0\leq{i}\leq{k}}$ at $T_j$ and invoking Theorem \ref{Atheorem.2.3} for each of the resulting portions yields
\begin{equation}\label{Alemma.4.3.equation.3}
d(A_n^-,A_n^+)\geq{d(A_n^-\contract{r},T_j\contract{r})+d(T_j\contract{r},A_n^+\contract{r})+3}\mbox{.}
\end{equation}

Call $s=n-q+1$. The portion of $A_n^-$ and $A_n^+$ close to vertex $a_q$ is depicted in Fig. \ref{Afigure.4.6} depending on whether $p\leq\lceil{n/2}\rceil$ or $p>\lceil{n/2}\rceil$. The arc introduced by the $j$-th flip along $(T_i)_{0\leq{i}\leq{k}}$ is depicted as a dotted line in this figure. 

One can see in Fig. \ref{Afigure.4.6} that the triangles incident to $\alpha_s$ in either $A_n^-\contract{r}$ and $A_n^+\contract{r}$ must intersect the arc with vertices $a_1$ and $a_p$ contained in $T_j\contract{r}$. More precisely, this property holds when $p\leq\lceil{n/2}\rceil$ because $q<p$ and because the boundary loop is above the dotted arc shown on the left of the figure. It holds when $p\leq\lceil{n/2}\rceil$ because $p<q<n$ and $1<s$ (right of the figure). As a consequence, the triangles incident to $\alpha_s$ in $A_n^-\contract{r}$ and $A_n^+\contract{r}$ cannot belong to $T_j\contract{r}$, and at least one flip is incident to $\alpha_s$ along any geodesic between $A_n^-\contract{r}$ and $T_j\contract{r}$ or between $T_j\contract{r}$ and $A_n^+\contract{r}$. Therefore, by inequality (\ref{Alemma.4.3.equation.3}) and Theorem \ref{Atheorem.2.3},
\begin{equation}\label{Alemma.4.3.equation.4}
d(A_n^-,A_n^+)\geq{d(A_n^-\contract{r}\contract{s},T_j\contract{r}\contract{s})+d(T_j\contract{r}\contract{s},A_n^+\contract{r}\contract{s})+5}\mbox{.}
\end{equation}

Since $A_n^-\contract{r}\contract{s}$ and $A_n^+\contract{r}\contract{s}$ are isomorphic to $A_{n-2}^-$ and $A_{n-2}^+$ by the same vertex relabeling, the result follows from (\ref{Alemma.4.3.equation.4}) and from the triangle inequality.
\end{proof}

We can now prove the main estimate.

\begin{theorem}\label{Atheorem.4.1}
For every integer $n$ greater than $2$,
$$
d(A_n^-,A_n^+)\geq\min(\{d(A_{n-1}^-,A_{n-1}^+)+3,d(A_{n-2}^-,A_{n-2}^+)+5\})\mbox{.}
$$
\end{theorem}
\begin{proof}
Assume that $n\geq3$ and consider a geodesic $(T_i)_{0\leq{i}\leq{k}}$ from $A_n^-$ to $A_n^+$. If at least $3$ flips are incident to $\alpha_n$ along it, then Theorem \ref{Atheorem.2.3} yields
$$
d(A_n^-,A_n^+)\geq{d(A_{n-1}^-,A_{n-1}^+)+3}\mbox{.}
$$

Indeed, as mentioned above, $A_n^-\contract{n}$ and $A_n^+\contract{n}$ are respectively isomorphic to $A_{n-1}^-$ and $A_{n-1}^+$ via the same vertex relabeling. Therefore in this case, the desired result holds. So we can assume in the remainder of the proof that at most two flips are incident to $\alpha_n$ along $(T_i)_{0\leq{i}\leq{k}}$. Further assume that the first flip incident to $\alpha_n$ along this geodesic is the $j$-th one. We review three cases, depending on which arc is introduced by this flip.

First assume that the $j$-th flip introduces an arc with vertices $a_0$ and $a_n$. This flip must be the one depicted in the left of Fig. \ref{Afigure.4.3}. Now consider a geodesic $(T'_i)_{0\leq{i}\leq{k'}}$ from $A_n^-\contract{1}$ to $T_j\contract{1}$, and a geodesic $(T''_i)_{j\leq{i}\leq{k''}}$ from $T_j\contract{1}$ to $A_n^+\contract{1}$. It follows from Lemma \ref{Alemma.4.1} and from Theorem \ref{Atheorem.2.3} that
\begin{equation}\label{Aequation.4.1.1}
k'+k''\leq{d(A_n^-,A_n^+)-3}\mbox{.}
\end{equation}

It can be seen on the left of Fig. \ref{Afigure.4.3} that the triangles incident to $\alpha_{n}$ in $A_n^-\contract{1}$, in $T_j\contract{1}$, and in $A_n^+\contract{1}$ are pairwise distinct. As a consequence at least one flip must be incident to $\alpha_n$ along each of the geodesics $(T'_i)_{0\leq{i}\leq{k'}}$ and $(T''_i)_{j\leq{i}\leq{k''}}$. In this case, Theorem \ref{Atheorem.2.3} yields
$$
k'\geq{d(A_n^-\contract{1}\contract{n},T_j\contract{1}\contract{n})+1}\mbox{ and }k''\geq{d(T_j\contract{1}\contract{n},A_n^+\contract{1}\contract{n})+1}\mbox{.}
$$

By the triangle inequality, one obtains
\begin{equation}\label{Aequation.4.1.2}
k'+k''\geq{d(A_n^-\contract{1}\contract{n},A_n^+\contract{1}\contract{n})+2}\mbox{.}
\end{equation}

Since $A_n^-\contract{1}\contract{n}$ and $A_n^+\contract{1}\contract{n}$ are isomorphic to $A_{n-2}^-$ and $A_{n-2}^+$ via the same vertex relabeling, the desired result follows from inequalities (\ref{Aequation.4.1.1}) and (\ref{Aequation.4.1.2}).

Now assume that the $j$-th flip introduces an arc with vertices $a_1$ and $a_2$. It follows from Lemma \ref{Alemma.4.1} and from Theorem \ref{Atheorem.2.3} that
\begin{equation}\label{Aequation.4.4}
d(A_n^-,A_n^+)\geq{d(A_n^-\contract{1},A_n^+\contract{1})+4}\mbox{.}
\end{equation}

Observe that arc $\alpha_{n-1}$ is not incident to the same triangle in $A_n^-\contract{1}$ and in $A_n^+\contract{1}$. Therefore, there must be at least one flip incident to $\alpha_{n-1}$ along any geodesic between these triangulations. Therefore, by Theorem \ref{Atheorem.2.3},
\begin{equation}\label{Aequation.4.5}
d(A_n^-\contract{1},A_n^+\contract{1})\geq{d(A_n^-\contract{1}\contract{n-1},A_n^+\contract{1}\contract{n-1})+1}\mbox{.}
\end{equation}

As $A_n^-\contract{1}\contract{n-1}$ and $A_n^+\contract{1}\contract{n-1}$ are isomorphic to $A_{n-2}^-$ and $A_{n-2}^+$ via the same vertex relabeling, the result is obtained combining (\ref{Aequation.4.4}) and (\ref{Aequation.4.5}).

Finally, if the $j$-th flip introduces an arc with vertices $a_1$ and $a_p$, where $2<p<n$, then $n$ must be greater than $3$ and the result follows from Lemma \ref{Alemma.4.3}.
\end{proof}

As we now have a lower and upper bounds, we can conclude the following.

\begin{theorem}\label{Atheorem.4.2}
The diameter of $\mathcal{MF}(\Gamma_n)$ is $\lfloor{ \frac{5}{2} n }\rfloor-2$.\end{theorem}
\begin{proof}
Since $\Gamma_1$ has a unique triangulation up to homeomorphism, $\mathcal{MF}(\Gamma_1)$ has diameter $0$. Moreover, as can be seen in Fig. \ref{Afigure.2.1}, $\mathcal{MF}(\Gamma_2)$ has diameter $3$. The lower bound of $\lfloor{5n/2}\rfloor-2$ on the diameter of $\mathcal{MF}(\Gamma_n)$ therefore follows by induction from Theorem \ref{Atheorem.4.1}. Combining this lower bound with the upper bound provided by Theorem \ref{thm:gammaupper} completes the proof.
\end{proof}

\section{Lower bounds for $\Pi$}
\label{Asection.5}

We now turn our attention to triangulations of $\Pi$. We shall, for any $n\geq 1$, build two triangulations $B_n^-$ and $B_n^+$ in $\MF(\Pi_n)$ whose flip distance is $3n+K_\Pi$, where $K_\Pi$ does not depend on $n$. 
\begin{figure}
\begin{centering}
\includegraphics{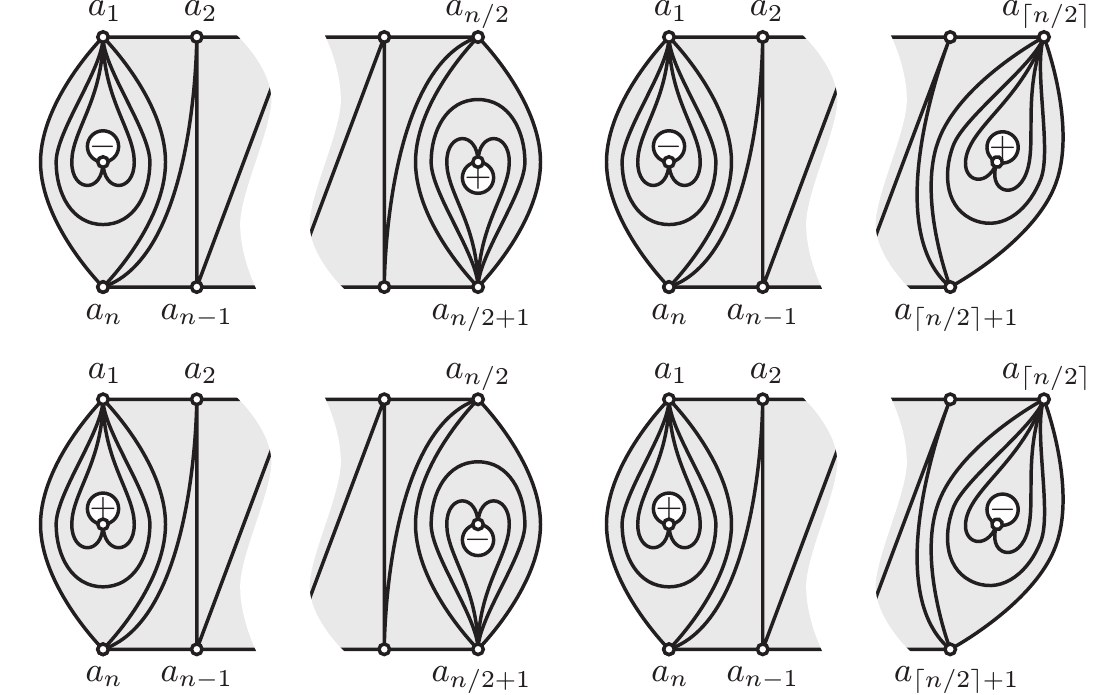}
\caption{Triangulations $B_n^-$ (top row) and $B_n^+$ (bottom row) depicted when $n$ is even (left) and odd (right). Vertices $a_-$ and $a_+$ are respectively labeled $-$ and $+$.}\label{Afigure.5.1}
\end{centering}
\end{figure}

First assume that $n$ is greater than $2$. Observe that $A_n^-$ has an ear in vertex $a_{\lfloor{n/2}\rfloor+1}$ (see Fig. \ref{Afigure.4.2}). One can transform $A_n^-$ into a triangulation that belongs to $\MF(\Pi_n)$ by placing a boundary loop $\alpha_+$ with a vertex $a_+$ in this ear and by re-triangulating the ear around the boundary loop as shown in the top of Fig. \ref{Afigure.5.1}, depending on the parity of $n$. Note that vertex $a_0$ and arc $\alpha_0$ are further relabeled by $a_-$ and $\alpha_-$ as shown in the figure. The resulting triangulation will be called $B_n^-$.

Similarly, consider the ear of $A_n^+$ in $a_1$. One can obtain a triangulation that belongs to $\MF(\Pi_n)$ by placing a boundary loop $\alpha_+$ with a vertex $a_+$ in this ear and by re-triangulating the pierced ear as shown in the bottom of Fig. \ref{Afigure.5.1}, depending on the parity of $n$. The resulting triangulation, wherein vertex $a_0$ and arc $\alpha_0$ have been relabelled $a_-$ and $\alpha_-$, will be called $B_n^+$.

When $1\leq{n}\leq2$, $B_n^-$ and $B_n^+$ will be the triangulations in $\MF(\Pi_n)$ depicted in Fig. \ref{Afigure.5.2}. Most of the section is devoted to proving the following inequality when $n\geq3$:
\begin{equation}\label{Aequation.5.1}
d(B_n^-,B_n^+)\geq\min(\{d(B_{n-1}^-,B_{n-1}^+)+3,d(B_{n-2}^-,B_{n-2}^+)+6\})\mbox{.}
\end{equation}

The proof consists in finding a geodesic between $B_n^-$ and $B_n^+$ within which at least a certain number of flips (typically three) are incident to given arcs. Using Theorem \ref{Atheorem.2.3} with well chosen vertex deletions will then result in (\ref{Aequation.5.1}). These deletions will be the same as in the case of triangulations $A_n^-$ and $A_n^+$. Observe in particular that, when $n\geq2$, the same vertex relabeling sends $B_n^-\contract{n}$ and $B_n^+\contract{n}$ to respectively $B_{n-1}^-$ and $B_{n-1}^+$.
\begin{figure}
\begin{centering}
\includegraphics{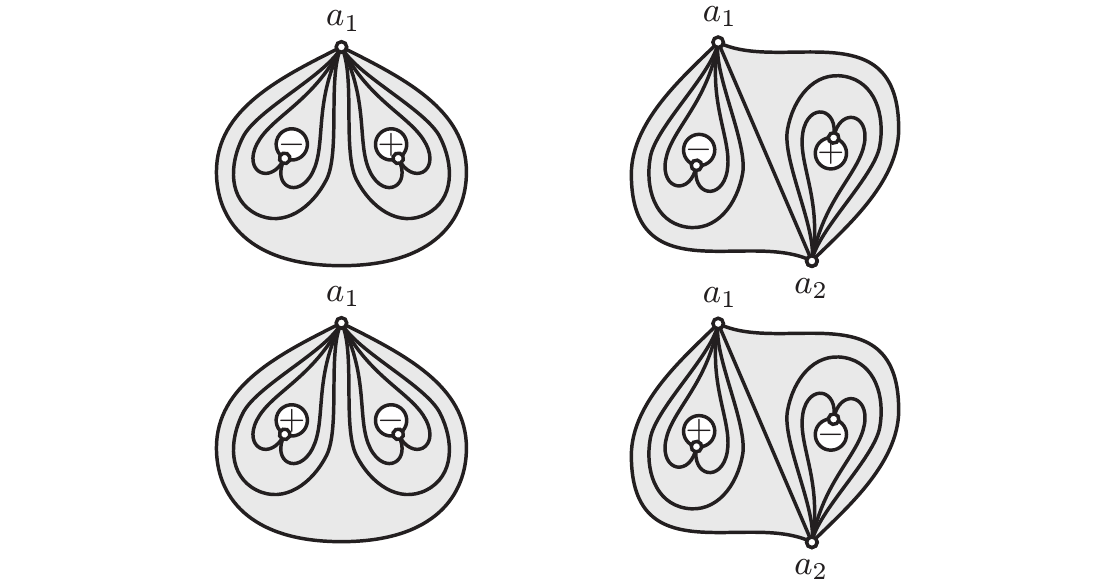}
\caption{Triangulations $B_n^-$ (top row) and $B_n^+$ (bottom row) depicted when $n=1$ (left) and when $n=2$ (right). Vertices $a_-$ and $a_+$ are respectively labeled $-$ and $+$.}\label{Afigure.5.2}
\end{centering}
\end{figure}
Moreover, if $n\geq3$ and if $i$ and $j$ are two integers so that $1\leq{i}<n$ and $j\in\{n-i,n-i+1\}$, then another vertex relabeling sends $B_n^-\contract{i}\contract{j}$ and $B_n^+\contract{i}\contract{j}$ to respectively $B_{n-2}^-$ and $B_{n-2}^+$.

\subsection{When an ear is found along a geodesic}

In this subsection, geodesics between $B_n^-$ and $B_n^+$ along which some triangulation has an ear are considered. Ears in $a_1$ and in $a_n$ are first reviewed separately. The following lemma deals with the case of an ear in $a_1$. Note that, by symmetry, this also settles the case of an ear in $a_{\lfloor{n/2}\rfloor+1}$.

\begin{lemma}\label{Alemma.5.1}
Let $n$ be an integer greater than $1$ and $(T_i)_{0\leq{i}\leq{k}}$ a geodesic between $B_n^-$ and $B_n^+$. If there exists an integer $j$ so that $0\leq{j}\leq{k}$ and $T_j$ has an ear in $a_1$, then
$$
d(B_n^-,B_n^+)\geq{d(B_{n-1}^-,B_{n-1}^+)+4}\mbox{.}
$$
\end{lemma}
\begin{proof}
Assume that $T_j$ has an ear in $a_1$ for some integer $j\in\{1, ..., k\}$. Call this ear $t$, and call $t^-$ the triangle incident to $\alpha_n$ in $B_n^-$. At least two of the first $j$ flips along $(T_i)_{0\leq{i}\leq{k}}$ must be incident to $\alpha_n$. Indeed, the unique such flip would otherwise replace triangle $t^-$ by $t$. This flip would then simultaneously remove two edges of $t^-$ (see the sketch of $B_n^-$ on the left Fig. \ref{Afigure.5.1}), which is impossible. By symmetry, at least two of the last $k-j$ flips along path $(T_i)_{0\leq{i}\leq{k}}$ must be incident to $\alpha_n$. Hence, at least four such flips are found along $(T_i)_{0\leq{i}\leq{k}}$, and Theorem \ref{Atheorem.2.3} yields
$$
d(B_n^-,B_n^+)\geq{d(B_n^-\contract{n},B_n^+\contract{n})+4}\mbox{.}
$$

Since an isomorphism sends $B_n^-\contract{n}$ and $B_n^+\contract{n}$ to $B_{n-1}^-$ and $B_{n-1}^+$ via the same vertex relabeling, the lemma is proven.
\end{proof}

The next lemma deals with the case of an an ear in $a_n$. By symmetry this also settles the case of an ear in $a_{n/2}$ when $n$ is even and in $a_{\lceil{n/2}\rceil+1}$ when $n$ is odd.
 
\begin{lemma}\label{Alemma.5.2}
Let $n$ be an integer greater than $2$ and $(T_i)_{0\leq{i}\leq{k}}$ a geodesic between $B_n^-$ and $B_n^+$. If there exists an integer $j$ so that $0\leq{j}\leq{k}$ and $T_j$ has an ear in $a_n$, then
$$
d(B_n^-,B_n^+)\geq\min(\{d(B_{n-1}^-,B_{n-1}^+)+3,d(B_{n-2}^-,B_{n-2}^+)+6\})\mbox{.}
$$
\end{lemma}
\begin{proof}
Assume that $T_j$ has an ear in $a_n$ for some integer $j\in\{1, ..., k\}$. One can see in Fig. \ref{Afigure.5.1} that the triangles of $B_n^-$ incident to arcs $\alpha_{n-1}$ and $\alpha_n$ do not have a common edge. Therefore, it follows from Lemma \ref{Alemma.2.75} that at least two of the first $j$ flips along $(T_i)_{0\leq{i}\leq{k}}$ are incident to $\alpha_r$ for some $r\in\{n-1,n\}$. By symmetry, the triangles of $B_n^+$ incident to arcs $\alpha_{n-1}$ and $\alpha_n$ do not have a common edge, and according to  the same lemma, at least two of the last $k-j$ flips along $(T_i)_{0\leq{i}\leq{k}}$ are incident to $\alpha_s$ for some $s\in\{n-1,n\}$.

Since the triangles incident to $\alpha_n$ in $B_n^-$ and in $B_n^+$ are distinct from the ear in $a_n$, at least one of the first $j$ flips and at least one of the last $k-j$ flips along $(T_i)_{0\leq{i}\leq{k}}$ are incident to $\alpha_n$. Hence, if $r$ or $s$ is equal to $n$, then at least three flips along this geodesic are incident to $\alpha_n$. In this case, the desired result follows from Theorem \ref{Atheorem.2.3} because $B_n^-\contract{n}$ and $B_n^+\contract{n}$ are isomorphic to respectively $B_{n-1}^-$ and $B_{n-1}^+$ via the same vertex relabeling.

Now assume that $r$ and $s$ are both equal to $n-1$. In this case, at least four flips along path $(T_i)_{0\leq{i}\leq{k}}$ are incident to $\alpha_{n-1}$ and Theorem \ref{Atheorem.2.3} yields:
\begin{equation}\label{Aequation.5.2}
d(B_n^-,B_n^+)\geq{d(B_n^-\contract{n-1},B_n^+\contract{n-1})+4}\mbox{.}
\end{equation}

Denote by $t^-$ and $t^+$ the triangles incident to edge $\alpha_1$ in respectively $B_n^-{\contract}n-1$ and $B_n^+{\contract}n-1$. One can see using Fig. \ref{Afigure.5.1} that these two triangles separate the two boundary loops in opposite ways.
\begin{figure}
\begin{centering}
\includegraphics{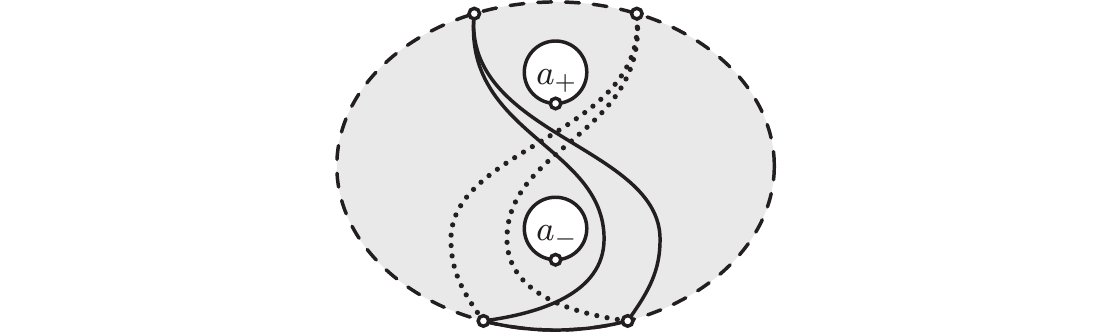}
\caption{No flip can replace triangle $t^-$ (solid lines) by triangle $t^+$ (dotted lines) because a such flip would simultaneously remove two edges of $t^-$.}\label{Afigure.5.3}
\end{centering}
\end{figure}
As shown in Fig. \ref{Afigure.5.3}, a single flip cannot exchange $t^-$ and $t^+$. Hence, at least two flips are incident to $\alpha_1$ along any geodesic between $B_n^-{\contract}n-1$ and $B_n^+{\contract}n-1$, and according to Theorem \ref{Atheorem.2.3},
\begin{equation}\label{Aequation.5.3}
d(B_n^-\contract{n-1},B_n^+\contract{n-1})\geq{d(B_n^-\contract{n-1}\contract{1},B_n^+\contract{n-1}\contract{1})+2}\mbox{.}
\end{equation}

Since $B_n^-{\contract}n-1\contract{1}$ and $B_n^+{\contract}n-1\contract{1}$ are isomorphic to respectively $B_{n-2}^-$ and $B_{n-2}^+$ via the same vertex relabeling, combining (\ref{Aequation.5.2}) with (\ref{Aequation.5.3}) completes the proof.
\end{proof}

Using the last two lemmas, one finds that if $n\geq3$ and if some triangulation along any geodesic between $B_n^-$ and $B_n^+$ has an ear, then the desired inequality holds.

\begin{theorem}\label{Atheorem.5.1}
Let $n$ be an integer greater than $2$ and $(T_i)_{0\leq{i}\leq{k}}$ a geodesic beteween $B_n^-$ and $B_n^+$. If there exists an integer $j$ so that $0\leq{j}\leq{k}$ and $T_j$ has an ear, then
$$
d(B_n^-,B_n^+)\geq\min(\{d(B_{n-1}^-,B_{n-1}^+)+3,d(B_{n-2}^-,B_{n-2}^+)+6\})\mbox{.}
$$
\end{theorem}
\begin{proof}
Consider an integer $j$ so that $0\leq{j}\leq{k}$ and assume that $T_j$ has an ear in vertex $a_p$ where $1\leq{p}\leq{n}$. If $p\in\{1,n\}$, then the desired result follows from Lemma \ref{Alemma.5.1} or from Lemma \ref{Alemma.5.2}. If $q\in\{\lceil{n/2}\rceil,\lceil{n/2}\rceil+1\}$, these two lemmas also provide the desired result because of the symmetries of $B_n^-$ and $B_n^+$. It is assumed in the remainder of the proof that $p$ does not belong to $\{1,\lceil{n/2}\rceil,\lceil{n/2}\rceil+1,n\}$.

Denote $q=n-p+1$. The portion of triangulation $B_n^-$ placed between edges $\alpha_{p-1}$, $\alpha_p$ and $\alpha_q$ is depicted on the left of Figure \ref{Afigure.5.4}. Note that, if one cuts geodesic $(T_i)_{0\leq{i}\leq{k}}$ at triangulation $T_j$, then Lemma \ref{Alemma.2.75} can be invoked for each of the resulting portions. Doing so, we find that either $\alpha_{p-1}$ and $\alpha_p$ are both incident to exactly $3$ flips along this geodesic or one of these arcs is incident to at least $4$ flips along it.

First assume that at least $4$ flips are incident to $\alpha_r$ along $(T_i)_{0\leq{i}\leq{k}}$, where $r$ is equal to $p-1$ or to $p$. Denote by $t^-$ and $t^+$ the triangles incident to arc $\alpha_q$ in respectively $B_n^-{\contract}r$ and $B_n^+{\contract}r$. One can see using Fig. \ref{Afigure.5.1} that these two triangles separate the two boundary loops in opposite ways. As shown in Fig. \ref{Afigure.5.3}, a single flip cannot exchange $t^-$ and $t^+$. Hence, at least two flips are incident to $\alpha_q$ along any geodesic between $B_n^-{\contract}r$ and $B_n^+{\contract}r$. Hence, invoking Theorem \ref{Atheorem.2.3} twice yields
$$
d(B_n^-,B_n^+)\geq{d(B_n^-\contract{r}\contract{q},B_n^+\contract{r}\contract{q})+6}\mbox{.}
$$
\begin{figure}
\begin{centering}
\includegraphics{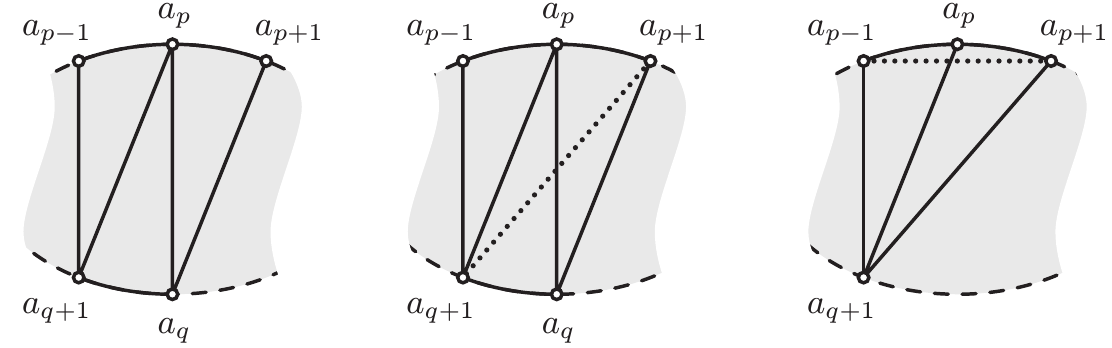}
\caption{The portion of triangulation $B_n^-$ placed between arcs $\alpha_{p-1}$, $\alpha_p$, and $\alpha_q$ (left), and the $i$-th flip along geodesic $(T_i)_{0\leq{i}\leq{k}}$ used in the proof of Theorem \ref{Atheorem.5.1}, with $i=l$ (center) and $i=j$ (right). Triangulation $T_{i-1}$ is shown in solid lines and the introduced edge is dotted.}\label{Afigure.5.4}
\end{centering}
\end{figure}

Since $B_n^-{\contract}r\contract{q}$ and $B_n^+{\contract}r\contract{q}$ are isomorphic to respectively $B_{n-2}^-$ and $B_{n-2}^+$ via the same vertex relabeling, the theorem is proven in this case.

Now assume that exactly three flips are incident to either $\alpha_{p-1}$ and $\alpha_p$ along $(T_i)_{0\leq{i}\leq{k}}$. Note that at least one of the first $j$ flips and at least one of the last $k-j$ flips along $(T_i)_{0\leq{i}\leq{k}}$ must be incident to either of these edges. Thanks to the symmetries of $B_n^-$ and $B_n^+$, we can assume without loss of generality that exactly one of the first $j$ flips and two of the last $k-j$ along $(T_i)_{0\leq{i}\leq{k}}$ are incident to $\alpha_{p-1}$. Then, by Lemma \ref{Alemma.2.75}, exactly two of the first $j$ flips and exactly one of the last $k-j$ along $(T_i)_{0\leq{i}\leq{k}}$ are incident to $\alpha_p$.

It is further assumed without loss of generality that the $j$-th flip along $(T_i)_{0\leq{i}\leq{k}}$ introduces the ear in $a_p$. This flip is then both the first flip incident to $\alpha_{p-1}$ and the second flip incident to $\alpha_p$ along the geodesic. In particular, this flip must replace the triangle of $B_n^-$ incident to $\alpha_{p-1}$ by the ear in $a_p$, as shown in the right of Fig. \ref{Afigure.5.4}. Now assume that the first flip incident to $\alpha_p$ along $(T_i)_{0\leq{i}\leq{k}}$ is the $l$-th one. Since there is no other such flip among the first $j-1$ flips along the geodesic, this flip must be the one shown in the center of Fig. \ref{Afigure.5.4}.

Consider a geodesic $(T'_i)_{0\leq{i}\leq{k'}}$ from $B_n^-\contract{p-1}$ to $T_l\contract{p-1}$, and a geodesic $(T''_i)_{j\leq{i}\leq{k''}}$ from $T_l\contract{p-1}$ to $B_n^+\contract{p-1}$. Since three flips are incident to $\alpha_{p-1}$ along $(T_i)_{0\leq{i}\leq{k}}$, it follows from Theorem \ref{Atheorem.2.3} that
\begin{equation}\label{Aequation.5.Th1}
k'+k''\leq{d(B_n^-,B_n^+)-3}\mbox{.}
\end{equation}

Observe that the triangles incident to $\alpha_q$ in $B_n^-\contract{p-1}$ and in $T_l\contract{p-1}$ are distinct. Hence, at least one flip is incident to $\alpha_q$ along $(T'_i)_{0\leq{i}\leq{k'}}$ and by Theorem \ref{Atheorem.2.3},
\begin{equation}\label{Aequation.5.4}
k'\geq{d(B_n^-\contract{p-1}\contract{q},T_l\contract{p-1}\contract{q})+1}\mbox{.}
\end{equation}

Now denote by $t^-$ and $t^+$ the triangles incident to edge $\alpha_q$ in respectively $T_l{\contract}p-1$ and $B_n^+{\contract}p-1$. By construction $t^-$ and $t^+$ separate the two boundary loops in opposite ways. As shown in Fig. \ref{Afigure.5.3}, a single flip cannot exchange $t^-$ and $t^+$. Hence, at least two flips are incident to $\alpha_q$ along $(T''_i)_{j\leq{i}\leq{k''}}$, and Theorem \ref{Atheorem.2.3} yields
\begin{equation}\label{Aequation.5.5}
k''\geq{d(T_l\contract{p-1}\contract{q},B_n^+\contract{p-1}\contract{q})+2}\mbox{.}
\end{equation}

By the triangle inequality, (\ref{Aequation.5.4}) and (\ref{Aequation.5.5}) yield
\begin{equation}\label{Aequation.5.Th2}
k'+k''\geq{d(B_n^-\contract{p-1}\contract{q},B_n^+\contract{p-1}\contract{q})+3}\mbox{.}
\end{equation}

Since $B_n^-\contract{s}\contract{q}$ and $B_n^+\contract{s}\contract{q}$ are isomorphic to $B_{n-2}^-$ and $B_{n-2}^+$ by the same vertex relabeling, the desired inequality is obtained combining (\ref{Aequation.5.Th1}) and (\ref{Aequation.5.Th2}).
\end{proof}

\subsection{When no ear is found along a geodesic}

We call a geodesic between $B_n^-$ and $B_n^+$ \emph{earless} if none of the triangulations along this geodesic has an ear. We will first show that under mild conditions, one always finds two particular triangulations along any such geodesics. These triangulations are sketched in Fig. \ref{Afigure.5.6}. The triangulation shown in the top row of this figure will be called $C_n^-(p)$. Note that $a_p$ is the privileged boundary vertex that is also a vertex of the triangle of $C_n^-(p)$ incident to arc $\alpha_+$. Further note that $C_n^-(p)$ is sketched separately when $p>\lceil{n/2}\rceil$ (left) and when $p\leq\lceil{n/2}\rceil$ (right). 
\begin{figure}
\begin{centering}
\includegraphics{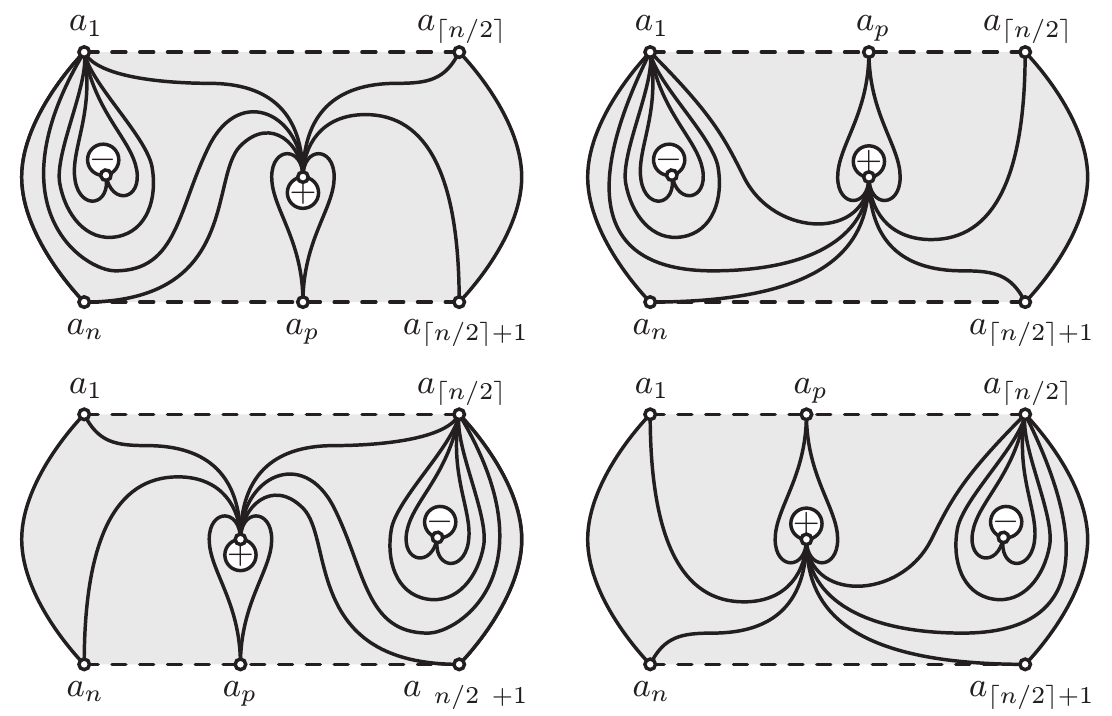}
\caption{Sketch of $C_n^-(p)$ (top) and $C_n^+(p)$ (bottom) when $p>\lceil{n/2}\rceil$ (left) and when $p\leq\lceil{n/2}\rceil$ (right). Not all the interior edges of these triangulations are shown. The omitted edges connect privileged boundary vertices to $a_+$.}\label{Afigure.5.6}
\end{centering}
\end{figure}
The triangulation shown in the bottom row of Fig. \ref{Afigure.5.6}, called $C_n^+(p)$ has a similar structure. In particular, $a_p$ is the vertex on the privileged boundary that is also a vertex of the triangle of $C_n^+(p)$ incident to arc $\alpha_+$. 

Observe that triangulations $C_n^-(p)$ and $C_n^+(p)$ do not have an ear. In fact, if at most two flips are incident to either $\alpha_n$ and $\alpha_{\lceil{n/2}\rceil}$ along an earless geodesic between $B_n^-$ and $B_n^+$, then these two triangulations are necessarily both found along this geodesic for appropriate values of $p$.

In order to prove this, the following lemma is needed:

\begin{lemma}\label{Alemma.5.3}
Let $n$ be an integer greater than $2$. If at most $2$ flips are incident to $\alpha_n$ along an earless geodesic from $B_n^-$ to $B_n^+$, then the first flip incident to $\alpha_n$ along this geodesic either introduces an arc with vertices $a_-$ and $a_n$ or an arc with vertices $a_1$ and $a_+$.
\end{lemma}
\begin{proof}
Consider a geodesic $(T_i)_{0\leq{i}\leq{k}}$ from $B_n^-$ to $B_n^+$ and assume that at most $2$ flips are incident to $\alpha_n$ along it. Further assume that the first flip incident to $\alpha_n$ along this geodesic is the $j$-th one. If this flip removes the loop edge of $B_n^-$ at vertex $a_1$, then it necessarily introduces the arc with vertices $a_-$ and $a_n$, as shown on the left of Fig. \ref{Afigure.5.7}. It is therefore assumed in the remainder of the proof that this flip removes the interior arc of $B_n^-$ with vertices $a_1$ and $a_n$. In this case, the introduced arc has vertices $a_1$ and $a_p$ where $1<p<n$ or $p=+$. It will be shown indirectly that $a_p$ is necessarily vertex $a_+$.

Assume that $1<p<n$.
\begin{figure}
\begin{centering}
\includegraphics{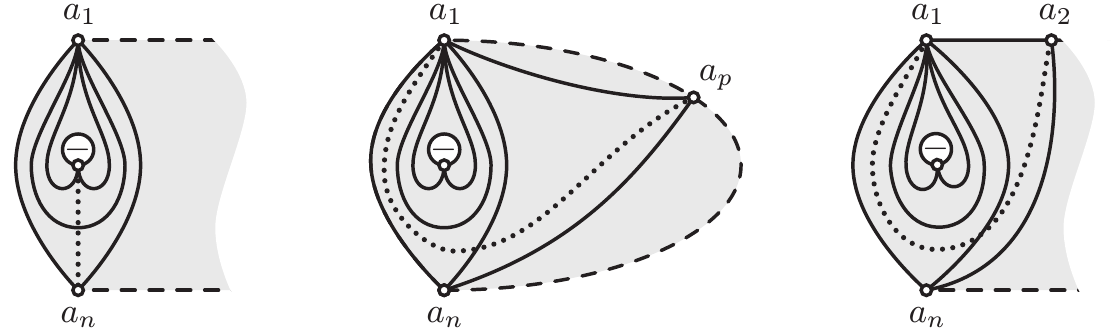}
\caption{The $j$-th flip along geodesic $(T_i)_{0\leq{i}\leq{k}}$ used in the proof of Lemma \ref{Alemma.5.3}. The arc introduced by this flip (dotted) has vertices $a_-$ and $a_n$ (left), or vertices $a_1$ and $a_p$ with $2\leq{p}<n$ (center and right).}\label{Afigure.5.7}
\end{centering}
\end{figure}
One can see in the center of Fig. \ref{Afigure.5.7} that in this case, $T_j$ induces a triangulation $U$ in the portion $\Sigma$ of $\Pi_n$ bounded by the dotted arc and by arcs $\alpha_p$, ..., $\alpha_n$. This triangulation cannot be a triangulation of a disc. Indeed, otherwise, one of the ears of $U$ would be an ear of $T_j$. This shows that the boundary loop with vertex $a_+$ must be a boundary of $\Sigma$. In this case, the $j$-th flip along $(T_i)_{0\leq{i}\leq{k}}$ must be the one shown in the right of Fig. \ref{Afigure.5.7}. Indeed, $T_j$ would otherwise induce a triangulation of a disc in the portion $\Pi_n$ bounded by arcs $\alpha_1$, ..., $\alpha_{p-1}$ and by the arc with vertices $a_1$ and $a_p$ shown in the center of the figure. This triangulation would then share one of its ears with $T_j$.

Finally, let $t^-$ and $t^+$ be the triangles incident to $\alpha_n$ in respectively $T_j$ and $B_n^+$. As the $j$-th flip along $(T_i)_{0\leq{i}\leq{k}}$ is the one shown in the right of Fig. \ref{Afigure.5.7}, $t^-$ and $t^+$ separate the two boundary loops in opposite ways. As shown in Fig. \ref{Afigure.5.3}, a single flip cannot exchange these triangles. Hence, at least two flips of the last $k-j$ flips must be incident to $\alpha_n$ along $(T_i)_{0\leq{i}\leq{k}}$, and at least three such flips are found along this geodesic, a contradiction.
\end{proof}

\begin{lemma}\label{Alemma.5.4}
For $n\geq3$, consider an earless geodesic $(T_i)_{0\leq{i}\leq{k}}$ from $B_n^-$ to $B_n^+$. If at most $2$ flips are incident to either $\alpha_n$ and $\alpha_{\lceil{n/2}\rceil}$ along this geodesic, then there exist an earless geodesic $(T'_i)_{0\leq{i}\leq{k}}$ from $B_n^-$ to $B_n^+$ and four integers $p^-$, $p^+$, $j^-$, and $j^+$ so that $j^-\leq{j^+}$,  and triangulations $T'_{j^-}$ and $T'_{j^+}$ are respectively equal to $C_n^-(p^-)$ and $C_n^+(p^+)$.
\end{lemma}
\begin{proof}
Assume that at most $2$ flips are incident to either $\alpha_n$ and $\alpha_{\lceil{n/2}\rceil}$ along $(T_i)_{0\leq{i}\leq{k}}$. In this case, exactly $2$ flips are incident to $\alpha_n$ along this geodesic. Indeed, otherwise the unique such flip would have to remove two arcs simulatenously as shown in Fig. \ref{Afigure.5.3}. Assume that the first flip incident to $\alpha_n$ along $(T_i)_{0\leq{i}\leq{k}}$ is the $j^-$-th one.

Call $t^-$ the triangle of $T_{j^-}$ incident to $\alpha_n$. From there on, $t^-$ remains incident to $\alpha_n$  in the triangulations visited by the geodesic until the second flip incident to $\alpha_n$ removes it. Moreover, according to Lemma \ref{Alemma.5.3}, the vertices of $t^-$ are $a_1$, $a_n$, and either $a_-$ or $a_+$. Thanks to the symmetries of $B_n^-$ and $B_n^+$, one can assume that this vertex is $a_+$. Indeed, if $a_-$ is a vertex of $t^-$, then exchanging the labels of $a_-$ and $a_+$ and inversing the direction of geodesic $(T_i)_{0\leq{i}\leq{k}}$ results in a geodesic from $B_n^-$ to $B_n^+$ whose first flip introduces an arc with vertices $a_1$ and $a_+$.
\begin{figure}[b]
\begin{centering}
\includegraphics{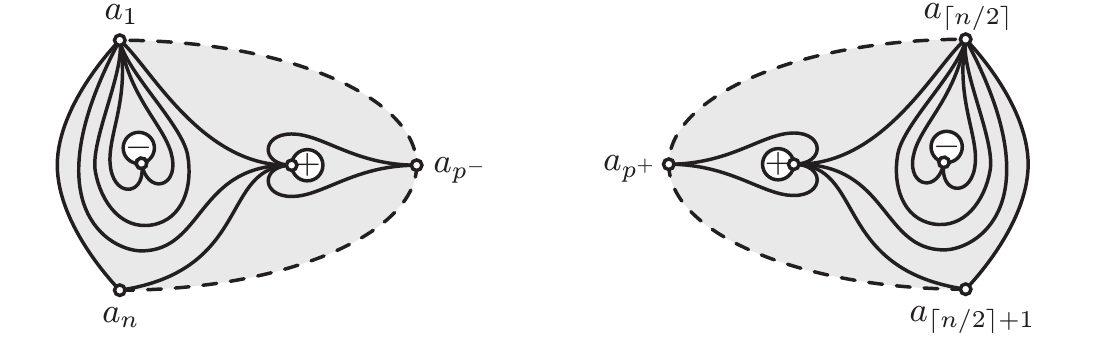}
\caption{Sketch of triangulations $T_{j^-}$ (left) and $T_{j^+}$ (right).}\label{Afigure.5.8}
\end{centering}
\end{figure}

According to this construction, triangulation $T_{j^-}$ must be as sketched on the left of Fig. \ref{Afigure.5.8}. Note in particular that the triangle incident to $\alpha_+$ is represented in this figure, the vertex of this triangle distinct from $a_+$ being privileged boundary vertex $a_{p^-}$. Moreover, not all the arc of $T_{j^-}$ are represented on the left of Fig. \ref{Afigure.5.8}. Since $T_{j^-}$ does not have an ear, then these missing arcs connect the privileged boundary vertices to $a_+$. There is only one way to place these missing arcs on the left of \ref{Afigure.5.8}. In particular $T_{j^-}$ is necessarily equal to $C_n^-(p^-)$.

Now observe that the triangle $t^+$ incident to $\alpha_{\lceil{n/2}\rceil}$ in $T_{j^-}$ (i.e. in $C_n^-(p^-)$) has vertices $a_{\lceil{n/2}\rceil}$, $a_{\lceil{n/2}\rceil+1}$, and $a_+$ (see top row in Fig. \ref{Afigure.5.6}). This triangle must be introduced by the first flip incident to $\alpha_{\lceil{n/2}\rceil}$ along $(T_i)_{0\leq{i}\leq{k}}$, and removed by the second flip incident to $\alpha_{\lceil{n/2}\rceil}$ along this geodesic. Say the latter flip transforms $T_{j^+}$ into $T_{j^++1}$. It replaces $t^+$ by the triangle incident to $\alpha_{\lceil{n/2}\rceil}$ in $B_n^+$. In particular, $T_{j^+}$ must already contain the triangle of $B_n^+$ incident to $\alpha_-$, whose vertices are $a_-$ and either $a_{\lceil{n/2}\rceil}$ or $a_{\lceil{n/2}\rceil+1}$ depending on the parity of $n$. Moreover, $t^-$ necessarily belongs to $T_{j^+}$. Indeed, the triangle of $B_n^+$ incident to $\alpha_n$ would otherwise belong to $T_{j^+}$, which is impossible because it intersects the interior of $t^+$.

Consider the portion $\Sigma$ of $\Pi_n$ bounded by arcs $\alpha_1$, ..., $\alpha_{\lceil{n/2}\rceil-1}$, by the edge of $t^-$ with vertices $a_1$ and $a_+$, and by the edge of $t^+$ with vertices $a_{\lceil{n/2}\rceil}$ and $a_+$. Observe that $\alpha_-$ is a boundary arc of $\Sigma$. Since $t^-$ and $t^+$ both belong to $T_{j^+}$, the triangle of $T_{j^+}$ incident to $\alpha_-$ necessarily admits $a_{\lceil{n/2}\rceil}$ as a vertex. In this case, $T_{j^+}$ must be as shown on the right of Fig. \ref{Afigure.5.8}. As $T_{j^+}$ does not have an ear, it is necessarily equal to $C_n^+(p^+)$ and the proof is complete.
\end{proof}
\begin{figure}[b]
\begin{centering}
\includegraphics{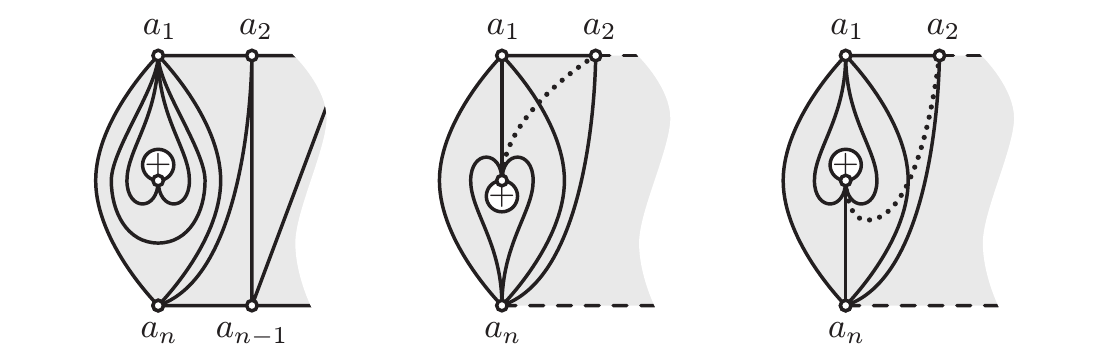}
\caption{Sketch of triangulation $B_m^+$ (left) and two possibility for the $j$-th flip along geodesic $(T_i)_{0\leq{i}\leq{k}}$ used in the proof of Lemma \ref{Alemma.5.5} (center and right). Triangulation $T_{j-1}$ is shown in solid lines and the introduced edge is dotted.}\label{Afigure.5.9}
\end{centering}
\end{figure}

\begin{lemma}\label{Alemma.5.5}
Let $n$ be an integer greater than $2$, and $p$ an integer so that $2\leq{p}\leq{n}$. If at most one flip is incident to $\alpha_1$ along some geodesic between $C_n^+(p)$ and $B_n^+$, then at least two flips are incident to $\alpha_n$ along this geodesic.
\end{lemma}
\begin{proof}
Consider a geodesic $(T_i)_{0\leq{i}\leq{k}}$ from $B_n^+$ to $C_n^+(p)$ and assume that at most one flip along this geodesic is incident to edge $\alpha_1$. In this case, there is exactly one such flip, say the $j$-th flip. Call $\beta$ the interior arc of $B_n^+$ with vertices $a_1$ and $a_n$. This arc is belong to $T_0$, ..., $T_{j-1}$ and it is removed by the flip that transforms $T_{j-1}$ into $T_j$. More precisely, this flip replaces $\beta$ by an arc with vertices $a_2$ and $a_+$. There are exactly two ways to do so, shown in Fig. \ref{Afigure.5.9}.

If the $j$-th flip along $(T_i)_{0\leq{i}\leq{k}}$ is the one shown in the center of Fig. \ref{Afigure.5.9}. Then at least two flips must have been performed within the portion $\Sigma$ of $\Pi_n$ bounded by $\beta$ and $\alpha_n$ earlier along the path (see $\mathcal{MF}(\Gamma_2)$ in Fig. \ref{Afigure.2.1}). By Proposition \ref{Aproposition.2.1}, these two flips are incident to $\alpha_n$ and the desired result holds.

If the $j$-th flip along $(T_i)_{0\leq{i}\leq{k}}$ is the one shown on the right of Fig. \ref{Afigure.5.9}. Then at least one of the earlier flips along the path modifies the triangulation within $\Sigma$. By Proposition \ref{Aproposition.2.1}, this flip is incident to $\alpha_n$. One can see on the right of Fig. \ref{Afigure.5.9} that the triangles incident to $\alpha_1$ and $\alpha_n$ in $T_j$ do not have a common edge. However, since $p$ is not equal to $1$, the triangles incident to these arcs in $C_n^+(p)$ have a common edge. Hence, at least one of the last $k-j$ flips along $(T_i)_{0\leq{i}\leq{k}}$ must be incident to $\alpha_n$, proving that at least two such flips are found along the geodesic.
\end{proof}

\begin{figure}[b]
\begin{centering}
\includegraphics{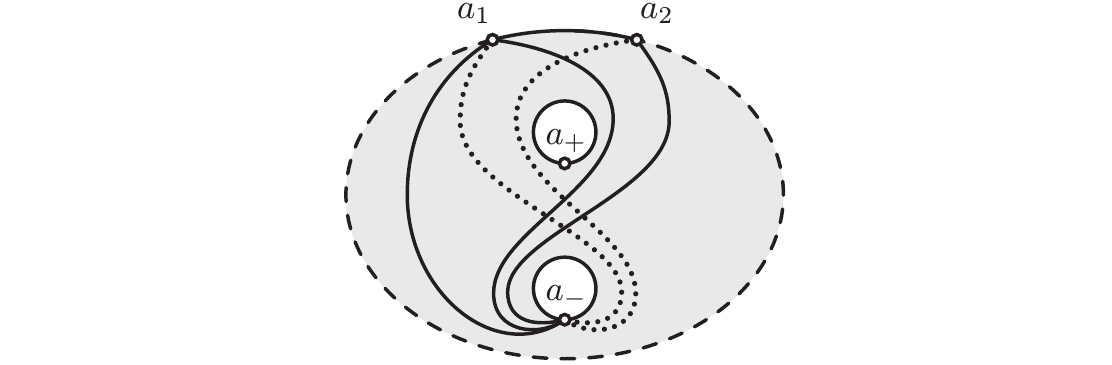}
\caption{The triangles incident to arc $\alpha_1$ in $C_n^-(\lceil{n/2}\rceil)$ (solid lines) and in $C_n^+(1)$ (dotted lines), and an edge of the triangle incident to $\alpha_n$ in these triangulations.}\label{Afigure.5.10}
\end{centering}
\end{figure}

\begin{lemma}\label{Alemma.5.7}
Let $n$ be an integer greater than $2$. If no flip is incident to $\alpha_n$ along a geodesic between $C_n^-(\lceil{n/2}\rceil)$ and $C_n^+(1)$, then at least two of its flips are incident to $\alpha_1$.
\end{lemma}
\begin{proof}
Assume that no flip is incident to $\alpha_n$ along a geodesic $(T_i)_{0\leq{i}\leq{k}}$ between $C_n^-(\lceil{n/2}\rceil)$ and $C_n^+(1)$. The triangles incident to $\alpha_1$ in $C_n^-(\lceil{n/2}\rceil)$ and $C_n^+(1)$ are depicted in Figure \ref{Afigure.5.10}, respectively in solid lines and in dotted lines. In this figure, the leftmost edge with vertices $a_1$ and $a_+$ is an edge of the triangle incident to $\alpha_1$ in both $C_n^-(\lceil{n/2}\rceil)$ and $C_n^+(1)$.

By hypothesis, this edge is never removed along geodesic $(T_i)_{0\leq{i}\leq{k}}$. Therefore, if exactly one of the flips along this geodesic is incident to edge $\alpha_1$, this flip must remove two edges of the triangle incident to $\alpha_1$ in $C_n^-(\lceil{n/2}\rceil)$, as shown in Figure \ref{Afigure.5.10}. Hence, there are at least two flips incident to arc $\alpha_1$ along $(T_i)_{0\leq{i}\leq{k}}$.
\end{proof}

\subsection{A lower bound on the diameter of $\mathcal{MF}(\Pi_n)$}

\begin{theorem}\label{Atheorem.5.2}
For any integer $n$ greater than $2$,
$$
d(B_n^-,B_n^+)\geq\min(\{d(B_{n-1}^-,B_{n-1}^+)+3,d(B_{n-2}^-,B_{n-2}^+)+6\})\mbox{.}
$$
\end{theorem}
\begin{proof}
Let $n$ be an integer greater than $2$. If one of the triangulations along any geodesic between $B_n^-$ and $B_n^+$ has an ear, then the desired result follows from Theorem \ref{Atheorem.5.1}, and it is assumed in the remainder of the proof that all the geodesic between $B_n^-$ and $B_n^+$ are earless. Moreover, if $p\in\{n,\lceil{n/2}\rceil\}$ and if at least three flips are incident to $\alpha_p$ along some geodesic between $B_n^-$ and $B_n^+$, the result follows from Theorem \ref{Atheorem.2.3} because $B_n^-\contract{p}$ and $B_n^+\contract{p}$ are respectively isomorphic to $B_{n-1}^-$ and $B_{n-1}^+$ via the same vertex relabeling. Hence, it will also be assumed that at most $2$ flips are incident to either $\alpha_n$ or $\alpha_{\lceil{n/2}\rceil}$ along any geodesic between $B_n^-$ and $B_n^+$. Under these assumptions, Lemma \ref{Alemma.5.4} provides an earless geodesic $(T'_i)_{0\leq{i}\leq{k}}$ from $B_n^-$ to $B_n^+$ and four integers $p^-$, $p^+$, $j^-$, and $j^+$ so that $j^-\leq{j^+}$, and $T'_{j^-}$ and $T'_{j^+}$ are respectively equal to $C_n^-(p^-)$, and $C_n^+(p^+)$.

First assume that $p^+>1$. Observe that the triangle incident to $\alpha_n$ in $C_n^+(p^+)$ is distinct from the triangles incident to this arc in $B_n^-$ and in $B_n^+$. As no more than $2$ flips are incident to $\alpha_n$ along $(T'_i)_{0\leq{i}\leq{k}}$, exactly one of the first $j^+$ flips and exactly one of the last $k-j^+$ flips along this geodesic are incident to $\alpha_n$. In this case, Lemma \ref{Alemma.5.5} states that at least two of the last $k-j^+$ flips along $(T'_i)_{0\leq{i}\leq{k}}$ are incident to $\alpha_1$. Now observe that the triangle incident to $\alpha_1$ in $C_n^-(p^-)$ is distinct from the triangles incident to this arc in $B_n^-$ and in $C_n^+(p^+)$. Hence at least two of the first $j^+$ flips along $(T'_i)_{0\leq{i}\leq{k}}$ are incident to $\alpha_1$, which proves that at least four such flips are found along this geodesic, and Theorem \ref{Atheorem.2.3} yields
\begin{equation}\label{Aequation.5.6}
d(B_n^-,B_n^+)\geq{d(B_n^-\contract1,B_n^+\contract1)+4}\mbox{.}
\end{equation}

We now show that this inequality still holds when $p^+=1$. Thanks to the symmetries of $B_n^-$ and $B_n^+$, the arguments in the last paragraph also prove (\ref{Aequation.5.6}) when $p^-$ is distinct from $\lceil{n/2}\rceil$. Now assume that $p^-=\lceil{n/2}\rceil$ and that $p^+=1$. In this case, according to Lemma \ref{Alemma.5.7}, at least two flips along $(T'_i)_{0\leq{i}\leq{k}}$ are incident to $\alpha_1$. Moreover, one can require that they take place between $C_n^-(p^-)$ and $C_n^+(p^+)$ along the geodesic. Now observe that the triangles of $B_n^-$ and $C_n^-(\lceil{n/2}\rceil)$ incident to $\alpha_1$ are distinct. Hence at least three of the first $p^+$ flips along $(T'_i)_{0\leq{i}\leq{k}}$ are incident to this arc. In addition, the triangles of $C_n^+(1)$ and $B_n^+$ incident to $\alpha_1$ are distinct. Therefore, at least one of the $k-p^+$ last flips along $(T'_i)_{0\leq{i}\leq{k}}$ is incident to $\alpha_1$, which proves that at least four such flips are found along this geodesic, and inequality (\ref{Aequation.5.6}) still holds in this case.

Finally observe that there must be at least two flips incident to $\alpha_{n-1}$ along any geodesic between $B_n^-\contract1$ and $B_n^+\contract1$. Indeed, the triangles $t^-$ and $t^+$ incident to $\alpha_{n-1}$ in these respective triangulations separate the two boundary loops in opposite ways. As shown in Fig. \ref{Afigure.5.3}, a single flip cannot exchange $t^-$ and $t^+$. Hence, at least two flips are incident to $\alpha_{n-1}$ along any geodesic between $B_n^-\contract1$ and $B_n^+\contract1$, and Theorem \ref{Atheorem.2.3} yields
\begin{equation}\label{Aequation.5.7}
d(B_n^-\contract1,B_n^+\contract1)\geq{d(B_n^-\contract1\contract{n-1},B_n^+\contract1\contract{n-1})+2}\mbox{.}
\end{equation}

Since $B_n^-\contract1\contract{n-1}$ and $B_n^+\contract1\contract{n-1}$ are isomorphic to $B_{n-2}^-$ and $B_{n-2}^+$ by the same vertex deletion, the result is obtained combining (\ref{Aequation.5.6}) and (\ref{Aequation.5.7}).
\end{proof}

Using this theorem, one obtains a lower bound on the diameter of $\mathcal{MF}(\Pi_n)$:

\begin{theorem}\label{Atheorem.5.3}
The diameter of $\mathcal{MF}(\Pi_n)$ is not less than $3n$.
\end{theorem}
\begin{proof}
One can see using Fig. \ref{Afigure.5.2} that at least three of the interior arcs of $A_1^-$ have to be removed in order to transform it into $A_1^+$. For instance, either all the arcs incident to $a_-$, or all the arcs incident to $a_+$ have to be removed. Hence:
\begin{equation}\label{Aequation.5.8}
d(B_1^-,B_1^+)\geq3\mbox{.}
\end{equation}

One can see on the same figure that transforming $A_2^-$ into $A_2^+$ requires to remove the arcs incident to $a_-$ and the arcs incident to $a_+$. As there are $6$ such arcs,
\begin{equation}\label{Aequation.5.9}
d(B_2^-,B_2^+)\geq6\mbox{.}
\end{equation}

The lower bound of $3n$ on the diameter of $\mathcal{MF}(\Pi_n)$ therefore follows by induction from Theorem \ref{Atheorem.5.2} and from inequalities (\ref{Aequation.5.8}) and (\ref{Aequation.5.9}).
\end{proof}

Observe that the flip distance of $B_1^-$ and $B_1^+$ is exactly $3$ (flipping all the arcs incident to $a_-$ provides a geodesic). The flip distance of $B_2^-$ and $B_2^+$ is, however equal to $7$ because all the interior arcs of $B_2^-$ have to be removed in order to transform this triangulation into $B_2^+$. This shows that the lower bound given by Theorem \ref{Atheorem.5.3} on $\diam(\mathcal{MF}(\Pi_n))$ is not sharp.

Finally, consider the two triangulations shown in Fig. \ref{Afigure.5.11}.
\begin{figure}
\begin{centering}
\includegraphics{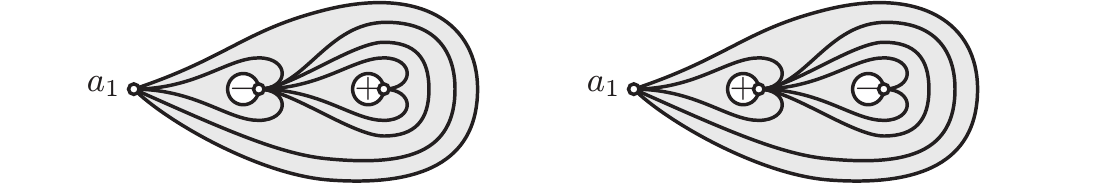}
\caption{Two triangulations in $\MF(\Pi_1)$ at flip distance at least $5$. Vertices $a_-$ and $a_+$ are respectively labeled $-$ and $+$.}\label{Afigure.5.11}
\end{centering}
\end{figure}
In order to transform the left triangulation into the right one, the three interior arcs incident to $a_1$ must be removed as well as the interior arc twice incident to with vertex $a_-$ and at least one of the arcs with vertices $a_-$ and $a_+$. As a consequence, the flip distance of these triangulations is at least $5$. This shows that even already when $n=1$, triangulations $B_n^-$ and $B_n^+$ are not maximally distant.

\section{Consequences and further questions}

As a first consequence of the above theorems, we prove the following. 

\begin{theorem}
Let $\Gamma$ be as defined previously. If $\Gamma \subset \Sigma$ is an essential embedding, then
$$
\lim_{n\to \infty} \frac{ \diam(\MFS) }{n} \geq \frac{5}{2}.
$$
\end{theorem}

\begin{proof}
If $\Gamma$ is embedded in $\Sigma$, there exists a surface $\Sigma'$ (possibly empty if $\Gamma = \Sigma$) such that gluing $\Sigma'$ and $\Gamma$ results in $\Sigma$. 

Now we take two diametrically opposite points $U$ and $V$ in $\MFC$ and send them to points in $\MFS$ by glueing a fixed triangulation of $\Sigma'$ to $U$ and to $V$. Denote by $U'$ and $V'$ the resulting triangulations of $\MFS$. We claim that
$$
d(U',V')=d(U,V).
$$

That the distance of $U'$ and $V'$ is at most that of $U$ and $V$ is obvious as any path in $\MFC$ can easily be emulated in $\MFS$. To see that $d(U',V')$ is at least $d(U,V)$ we will use Lemma \ref{lem:projlemma}. By the lemma, if two triangulations in $\FS$ have an arc or a set of arcs in common, then any geodesic between them conserves these arcs. Now, of course this property may no longer be true when one quotients by the group of homeomorphisms, but in this case it works. Indeed, as we consider homeomorphisms that preserve marked points, the isotopy class of a curve parallel to the privileged boundary curve is preserved by any such homeomorphism. This implies that the isotopy class of the embedding of the boundary loop of $\Gamma_n$ is also preserved. Thus there exists a geodesic between $U'$ and $V'$ such that all triangulations contain this arc. Any flip on the $\Sigma'$ side of the surface would be superfluous. Hence there is a geodesic that lies entirely in this natural copy of $\MFC$, and we are done.
\end{proof}

This theorem implies that the diameter growth rate for all filling surfaces is at least on the order of $5n/2$ except for the disc, the once punctured disc, and possibly for the filling surfaces of positive genus without interior vertices or non-privileged boundaries.
For example if $\Sigma$ is a disk with two unmarked points, then the diameter of the modular flip-graphs of $\Sigma$ grows like $5n/2$.

In fact there are multiple variations and consequences either of the above results or of the method of their proof. For example, one could try and emulate the method for the lower bounds of $\Pi$, but the combinatorics become more and more difficult to handle. There is reason to believe that increasing the number of marked boundary loops might increase the diameter of the underlying flip-graph. In the case of unmarked boundary loops, we can also expect some form of monotonicity in function of the topology. In fact we suspect that the following is true.

\begin{conjecture}
For any $\varepsilon>0$ there exists a $k_\varepsilon$ such that if $\Sigma$ is a surface with $k_\varepsilon$ marked boundary non privileged loops, the diameters of its flip-graphs satisfy
$$
\lim_{n\to \infty}  \frac{ \diam(\MFS) }{n} \geq 4 - \varepsilon.
$$
\end{conjecture}

In the unmarked case, we conjecture the following.

\begin{conjecture}
For any $\varepsilon>0$ there exists a $k_\varepsilon$ such that if $\Sigma$ is a surface with $k_\varepsilon$ unmarked boundary non privileged loops, the diameters of its flip-graphs satisfy
$$
\lim_{n\to \infty}  \frac{ \diam(\MFS) }{n} \geq 3 - \varepsilon.
$$
\end{conjecture}

There are many other questions that we feel could be interesting. A very basic one is to understand the growth of diameter of the flip-graph when $\Sigma$ is a torus (with a privileged boundary curve). Our methods in their current state are not able to say anything meaningful in this case.

Other more complicated variations of the above problems are for surfaces where we have multiple privileged boundary components and add vertices to several of them. We suspect that one can find very different diameter growths by sufficiently varying the problem. 

To conclude we now have examples of $\Sigma$ with $2n$, $\frac{5}{2}n $ and $3 n$ growth rate. This begs the question of classifying which numbers can appear as growth rates of these diameters. We suspect that the growth rates continue to change when the topology changes. More precisely we conjecture the following.

\begin{conjecture}
The number of topological types of filling surfaces with the same growth rate is finite. 
\end{conjecture}

\addcontentsline{toc}{section}{References}
\bibliographystyle{amsplain}
\bibliography{FlippinHell}

\providecommand{\bysame}{\leavevmode\hbox to3em{\hrulefill}\thinspace}
\providecommand{\MR}{\relax\ifhmode\unskip\space\fi MR }
\providecommand{\MRhref}[2]{%
  \href{http://www.ams.org/mathscinet-getitem?mr=#1}{#2}
}
\providecommand{\href}[2]{#2}
\begin{thebibliography}{10}

\bibitem{BridsonHaefliger1999}
M.~R. Bridson and A.~Haefliger, \emph{Metric spaces of non-positive curvature},
  Grundlehren der Mathematischen Wissenschaften [Fundamental Principles of
  Mathematical Sciences], vol. 319, Springer-Verlag, Berlin, 1999.

\bibitem{BrooksMakover2004}
R.~Brooks and E.~Makover, \emph{Random construction of {R}iemann surfaces}, J.
  Diff. Geom. \textbf{68} (2004), no.~1, 121--157.

\bibitem{DisarloParlier2014}
V.~Disarlo and H.~Parlier, \emph{The geometry of the flip graph and mapping
  class groups}, in preparation (2014).

\bibitem{FominShapiroThurston2008}
S.~Fomin, M.~Shapiro, and D.~Thurston, \emph{Cluster algebras and triangulated
  surfaces. {P}art {I}: {C}luster complexes}, Acta. Math. \textbf{201} (2008),
  83--146.

\bibitem{FominThurston2012}
S.~Fomin and D.~Thurston, \emph{Cluster algebras and triangulated surfaces.
  {P}art {II}: {L}ambda lengths}, arXiv:1210.5569 (2012).

\bibitem{FominZelevinsky2003}
S.~Fomin and A.~Zelevinsky, \emph{$\mbox{$Y$}$-systems and generalized
  associahedra}, Ann. Math. \textbf{158} (2003), 977--1018.

\bibitem{GelfandKapranovZelevinsky1990}
I.~M. Gel'fand, M.~M. Kapranov, and A.~V. Zelevinsky, \emph{Discriminants of
  polynomials of several variables and triangulations of {Newton} polyhedra},
  Leningrad Math. J. \textbf{2} (1990), 449--505.

\bibitem{KorkmazPapadopoulos2012}
M.~Korkmaz and A.~Papadopoulos, \emph{On the ideal triangulation graph of a
  punctured surface}, Ann. Inst. Fourier \textbf{62} (2012), no.~4, 1367--1382.

\bibitem{Lee1989}
C.~W. Lee, \emph{The associahedron and triangulations of the $n$-gon}, Eur. J.
  Comb. \textbf{10} (1989), 551--560.

\bibitem{DeLoeraRambauSantos2010}
J.~A.~{De} {Loera}, J.~Rambau, and F.~Santos, \emph{Triangulations: structures
  for algorithms and applications}, Algorithms and Computation in Mathematics,
  vol.~25, Springer, 2010.

\bibitem{Mosher1988}
L.~Mosher, \emph{Tiling the projective foliation space of a punctured surface},
  Trans. Am. Math. Soc. \textbf{306} (1988), no.~1, 1--70.

\bibitem{Penner1987}
R.~Penner, \emph{The decorated {T}eichm\"uller space of punctured surfaces},
  Comm. Math. Phys. \textbf{113} (1987), no.~2, 299--339.

\bibitem{Pournin2014}
L.~Pournin, \emph{The diameter of associahedra}, Adv. Math. \textbf{259}
  (2014), 13--42.

\bibitem{SleatorTarjanThurston1988}
D.~Sleator, R.~Tarjan, and W.~Thurston, \emph{Rotation distance,
  triangulations, and hyperbolic geometry}, J. Am. Math. Soc. \textbf{1}
  (1988), 647--681.

\bibitem{Stasheff1963}
J.~Stasheff, \emph{Homotopy associativity of {$H$-spaces}}, Trans. Am. Math.
  Soc. \textbf{108} (1963), 275--312.

\bibitem{Stasheff2012}
\bysame, \emph{How {I} `met' {Dov} {Tamari}}, Associahedra, Tamari Lattices and
  Related Structures, Progress in Mathematics, vol. 299, Birkh{\"a}user, 2012,
  pp.~45--63.

\bibitem{Tamari1954}
D.~Tamari, \emph{Monoides pr{\'e}ordonn{\'e}s et cha{\^i}nes de {Malcev}}, B.
  Soc. Math. Fr. \textbf{82} (1954), 53--96.

\end{thebibliography}

{\em Addresses:}\\
Department of Mathematics, University of Fribourg, Switzerland \\
LIAFA, Universit{\'e} Paris Diderot, Case 7014, 75205 Paris Cedex 13, France\\
{\em Emails:} \href{mailto:hugo.parlier@unifr.ch}{hugo.parlier@unifr.ch}, \href{mailto:lionel.pournin@liafa.univ-paris-diderot.fr}{lionel.pournin@liafa.univ-paris-diderot.fr}\\

\end{document}